\documentclass[a4paper,12pt]{amsart}

\pdfoutput=1
\usepackage{multicol}
\usepackage{multirow}
\usepackage{amsmath,latexsym,amsbsy,amssymb,fancyhdr}
\usepackage{array}
\usepackage{enumerate}
\usepackage{amsthm}
\usepackage{booktabs} 
\usepackage{multirow}
\usepackage{tabularx}
 \usepackage{epsfig} 

\usepackage{psfrag,graphicx}
\usepackage{epstopdf}
\usepackage[latin1]{inputenc}
\usepackage[colorlinks=true]{hyperref}
\hypersetup{urlcolor=blue, citecolor=blue, linkcolor=blue}

\newtheorem{Theorem}{Theorem}[section]
\newtheorem{Definition}[Theorem]{Definition} 
\newtheorem{Proposition}[Theorem]{Proposition}

\newtheorem{Lemma}[Theorem]{Lemma} 
\newtheorem{Assumption}[Theorem]{Assumptions}

\newtheorem{Algorithm}[Theorem]{Algorithm}
\newtheorem{Remark}[Theorem]{Remark}

\numberwithin{equation}{section}

\makeatletter
\def\@date{23 December 2015}
\let\insertdate\@date
\makeatother

\newcommand{\ang}[1]{\langle #1 \rangle}
\newcommand{\bb}[1]{\lbrace #1 \rbrace}
\newcommand{\co}{\operatorname{co}}

\newcommand{\epi}{\operatorname{epi}}
\newcommand{\sign}{\operatorname{sign}}
\newcommand{\gph}{\operatorname{gph}}

\newcommand{\di}{\operatorname{d}}
\newcommand{\diam}{\operatorname{diam}}

\newcommand{\dH}{\di_H}

\newcommand{\inter}{\operatorname{int}}
\newcommand{\norm}[1]{\left\|#1\right\|}
\newcommand{\Orem}[1]{\mathcal{O}(#1)}
\newcommand{\rk}{\operatorname{rank}}
\newcommand{\RSU}{\mathcal{R}_\leq}
\newcommand{\scal}[2]{\langle {#1},{#2} \rangle}
\newcommand{\Y}{\operatorname{Y}}

\newcommand{\R}{{\mathbb R}}
\newcommand{\N}{{\mathbb N}}

\usepackage{xcolor}
\definecolor{palegreen}{rgb}{0.2,0.6,0.2}

\author{Robert Baier}
\address[Robert Baier]{Mathematisches Institut,
Universit\"at Bayreuth,  
95440 Bayreuth, Germany}
\email{robert.baier@uni-bayreuth.de}

\author{Thuy T. T. Le}
\address[Thuy T.T. Le]{Universit\`a di Padova, Dipartimento di Matematica, via Trieste 63, 35121 Padova, Italy}
\email{lethienthuy@gmail.com}
\thanks{The second author is supported by a PhD fellowship for foreign students at the Universit\`a di Padova funded by Fondazione CARIPARO. This paper was developed while the second author was visiting the Department of Mathematics of the University of Bayreuth}
\begin{document}

\title{Construction of the Minimum Time Function \newline Via Reachable Sets of Linear Control Systems. \newline Part~I: Error Estimates}



\date{\today}
\begin{abstract}
The first part of this paper is devoted to introducing an approach to compute the approximate minimum time function of control problems which is based on reachable set approximation and uses arithmetic operations for convex compact sets. In particular, in this paper the theoretical justification of the proposed approach is restricted to a class of linear control systems.   The error estimate of the fully discrete reachable set is provided by employing the Hausdorff distance to the continuous-time reachable set. The detailed procedure solving the corresponding discrete set-valued problem is described. Under standard assumptions, by means of convex analysis and knowledge of the regularity of the true minimum time function, we estimate the error of its approximation. 
Numerical examples are included in the second part.
\end{abstract}
\maketitle

\markright{\hfill Construction of the minimum time function via reachable sets. Part 1 \hfill}
\markleft{\hfill R. Baier, T. T. T. Le \hfill}

\section{Introduction}

Reachable sets have attracted several mathematicians since longer times 
both in theoretical and in numerical analysis. One common definition collects
end points of feasible solutions of a control problem starting from a common inital set 
and reaching a point
\emph{up to} a given end time, the other definition is similar but prescribes a \emph{fixed}
end time in which the point is reached. The former definition automatically leads to a
monotone behavior of the reachable sets with respect to inclusion, since the reachable set
up to a given time is the union of reachable sets for a fixed time. 

Reachable sets with and without control constraints appear in control theory (e.g.~in
stability results),
in optimal control (e.g.~in analysis for robustness) and in set-valued analysis.
For reachable sets at a given end time of linear or nonlinear control problems,
properties like convexity for linear control problems 
at a given end time (due to Aumann and his study of Aumann's integral for set-valued
maps in \cite{Aum}), closedness and connectedness under weak assumptions for nonlinear
systems (see e.g.~\cite{Dav,AC}), \ldots\ are well-known. 
The Lipschitz continuity of reachable sets with respect to the initial value is also
established and is a result of the Filippov theorem which proves the existence of
neighboring solutions for Lipschitz systems. To mention one further result is the
density of solutions of the non-convexified control problem in the relaxed system
in which the right-hand side is convexified.

On the other hand reachable sets appear in many applications. They appear in natural
generalizations of differential equations with discontinuous right-hand side and hybrid
systems (e.g.~via the Filippov regularization in \cite{Fil}), in gradient inclusions
with maximally monotone right-hand side (see e.g.~\cite{AC}), as generalizations
of control problems (see e.g.~\cite{AF}), \ldots\ .
Many practical examples are mentioned in~\cite{ABS-P,AC,Alth,BGX,DL} and in references therein.

The approaches for the numerical computation of reachable sets mainly split into
two classes, those for reachable sets up to a given time and the other ones 
for reachable sets at a given end time. We will give here only exemplary references,
since the literature is very rich. There are methods based on overestimation and 
underestimation of reachable sets based on ellipsoids \cite{KV}, zonotopes \cite{Alth,GlGM} or
on approximating the reachable set with support functions resp.~supporting points 
 \cite{BL,KrK,lG,lGG,Alth}. Other popular and well-studied approaches involve level-set methods,
semi-Lagrangian schemes
and the computation of an associated Hamilton-Jacobi-Bellman equation, 
see~e.g.~\cite{BCD,BFZ,F,FF,G,MBT} or are based on the viability concept~\cite{ABS-P}
and the viability kernel algorithm~\cite{S-P}. Further methods \cite{BL,BLcham,BPhD,B,Cha} are set-valued generalizations of
quadrature methods and Runge-Kutta methods initiated by the works~\cite{D2F,V_int,DF,W,DV}.
Solvers for optimal control problems are another source for methods approximating
reachable sets, see~\cite{BBCG,BGX,GJ}.
In~\cite{DL,BPhD} a more detailed review of some methods up to 1994 
appeared, see also~\cite{BGX,BFZ} and the books and book chapters in~\cite{FF,KV,F} for a more recent overview and references therein.

Here, we will focus on set-valued quadrature methods and set-valued Runge-Kutta methods
with the help of support functions or supporting points,
since they do not suffer on the wrapping effect or on an exploding number of vertices
and the error of restricting computations only for finitely many directions 
can be easily estimated. Furthermore, they belong to the most efficient and fast methods
(see~\cite[Sec.~3.1]{Alth}, \cite[Chap.~9, p.~128]{lG}) for linear control problems
to which we restrict the computation of the minimum time function $T(x)$.
These methods enjoy an increasing attention also in neighboring research fields, 
e.g.~in the computation of viability kernels~\cite{MKMOD} or reachable sets 
for hybrid systems in~\cite{lG,lGG_hybrid} as well as in the computation of
interpolation of set-valued maps \cite{L},
Minkowski sums of convex sets (\cite{BPhD,SGJ})
as well as of the Dini, Michel-Penot and Mordukhovich subdifferentials~\cite{BF,BFR}.
We refer to~\cite{BPhD,BL,lG,lGG} (and references therein) for technical details on the numerical implementation, 
although we will lay out the main ideas of this approach for reader's convenience.

In optimal control theory the regularity of the minimum time functions is studied
intensively, see e.g.~in~\cite{CMW,CNN} and references therein.
For the error estimates in this paper it will be essential to single out example classes
for which the minimum time function is Lipschitz (no order reduction of the set-valued method) 
or H\"older-continuous with exponent $\frac{1}{2}$ (order reduction by the square root).
 
Minimum time functions are usually computed by solving
Hamilton-Jacobi-Bellman (HJB) equations and by the dynamic programming principle, see e.g.~\cite{BBZ,BFZ,BF1,BF2,BFS,CL,GL,MB}. 
In this approach, the minimal requirement on the regularity of $T(x)$ is the continuity, see e.g.~\cite{BF1,CL,GL}. The solution of a HJB equation with suitable boundary conditions gives immediately -- after a transformation -- the minimum time function and its level sets provide a description of the reachable sets. A natural question occurring is whether it is also possible to do the other way around, i.e.~reconstruct the minimum time function $T(x)$ if knowing the reachable sets. One of the attempts was done in \cite{BBZ,BFZ}, where the approach is based on PDE solvers and 
 on the reconstruction of the optimal control and solution via the value function.
On the other hand, our approach in this work is completely different. It is based on 
 very efficient quadrature methods for convex reachable sets as described in Section~3.
 
In this article we present a novel approach for calculating the minimum time function.
The basic idea is to use set-valued methods for approximating reachable sets at a given
end time with computations based on support functions resp.~supporting points.
By reversing the time and start from the convex target as initial set we compute
the reachable sets for times on a (coarser) time grid. 
Due to the strictly expanding condition for reachable sets,
the corresponding end time is assigned to all boundary points of the computed 
reachable sets. Since we discretize in time and in space (by choosing a finite number
of outer normals for the computation of supporting points), the vertices of the polytopes
forming the fully discrete reachable sets are considered as data points of an irregular
triangulated domain. On this simplicial triangulation, a piecewise linear approximation
yields a fully discrete approximation of the minimum time function. 

The well-known
interpolation error and the convergence results for the set-valued method can be applied
to yield an easy-to-prove error estimate by taking into account the regularity of the 
minimum time function. It requires at least the continuity and 
involves the maximal diameter of the simplices in the used
triangulation. A second error estimate is proved without explicitely assuming
the continuity of the minimum time function and depends only on the time interval 
between the computed (backward) reachable sets. The computation does not need
the nonempty interior of the target set in contrary to the Hamilton-Jacobi-Bellman approach,
for singletons the error estimate even improves. It is also able to compute discontinuous
minimum time functions, since the underlying set-valued method can also compute
lower-dimensional reachable sets. There is no explicit dependence of the algorithm
and the error estimates on the smoothness of optimal solutions or controls. 
These results 
are devoted to reconstructing discrete 
optimal trajectories which reach a set of supporting points from a given target 
for a class of linear control problems and also proving the convergence of discrete 
optimal controls by the use of nonsmooth and variational analysis. 
The main tool is Attouch's theorem that allows to benefit from the convergence of the discrete reachable
sets to the time-continuous one.

The plan of the article is as follows: in Section~2 we collect notations, definitions
and basic properties of convex analysis, set operations, reachable sets and the
minimum time function. The convexity of the reachable set for linear control problems
and the characterization of its boundary via the level-set of the minimum time function
is the basis for the algorithm formulated in the next section. We 
briefly introduce the reader to set-valued quadrature methods and Runge-Kutta methods
and their implementation and
discuss the convergence order for the fully discrete approximation of reachable sets at a given time both in time and in space.
In the next subsection we present the error estimate for the fully discrete minimum time function
which depends on the regularity of the continuous minimum time function and on the
convergence order of the underlying set-valued method. Another error estimate expresses
the error only on the time period between the calculated reachable sets. The last subsection discusses the construction of discrete optimal trajectories and convergence of discrete optimal controls.

Various accompaning examples can be found in the second part~\cite{BLp2}.

\section{Preliminaries}\label{sec:preliminaries}
In this  section we will recall some notations, definitions as well as basic knowledge of convex analysis and control theory for later use. Let $\mathcal{C}(\mathbb{R}^n)$ be the set of convex, compact,  nonempty  subsets of $\mathbb{R}^n$, 
 $\| \cdot \|$ be the Euclidean norm and $\scal{\cdot}{\cdot}$ the inner product in $\mathbb{R}^n$,
 $B_r(x_0)$ be  the closed (Euclidean) ball with radius $r>0$ centered at $x_0$  and $S_{n-1}$ be the unique sphere in $\mathbb{R}^n$. 
Let $A$ be a subset of $\mathbb{R}^n$, $M$ be an $n\times n$ real matrix, then $B_r(A):= \bigcup_{x\in A}B_r(x)$, $\norm{M}$ denotes the \emph{lub-norm}
 of $M$ with respect to $\|\cdot\|$, i.e.~the spectral norm.
The \emph{convex hull}, the \emph{boundary} and the \emph{interior} of a set $A$ are signified by $\co(A),\,\partial A,\, \inter(A)$ respectively. We define the support function, the supporting points in a given direction and the set arithmetic operations as follows.
\begin{Definition}
Let $A\in \mathcal{C}(\mathbb{R}^n),\, l \in  \R^n$. The \emph{support function} and the \emph{supporting face} of $A$ in the direction $l$ are defined as, respectively,
\begin{equation*}
\begin{aligned}
\delta^*(l,A)&:= \max_{x\in A} \,\scal{l}{x},\\
 \Y(l,A)&:=\bb{x\in A \colon  \scal{l}{x}=\delta^*(l,A)}.
\end{aligned}
\end{equation*}
An element of the supporting face is called \emph{supporting point}.
\end{Definition}

\begin{Definition}
Let $A,B \in \mathcal{C}(\mathbb{R}^n),\, \lambda \in \R,\,M\in \mathbb{R}^{m\times n}$. Then the \emph{scalar multiplication}, the \emph{image of a set under a linear map} and the \emph{Minkowski sum} are defined as follows:
\begin{align*}
\lambda A&:=\bb{\lambda a \colon a\in A},\\
 M A&:=\bb{M a \colon a\in A},\\
A+B&:=\bb{a+b \colon a\in A,\, b\in B}.
\end{align*}
\end{Definition}
In the following propositions we will recall known properties of the convex hull, the support function and the supporting points when applied to the set operations introduced above (see e.g.~\cite[Chap.~0]{AC}, \cite[Sec.~4.6, 18.2]{ABS-P}, \cite{BPhD,lG,Alth}).
Especially, the convexity of the arithmetic set operations becomes obvious.
\begin{Proposition}
Let $A,B\in \mathbb{R}^n,\,M\in \mathbb{R}^{m\times n}$ and $\lambda \in \mathbb{R}$.
Then,
\begin{equation*}
\begin{aligned}
\co(A+B)&=\co(A)+\co(B),\\
\co(\lambda A)&=\lambda \co(A),\\
\co(MA)&=M \co(A).
\end{aligned}
\end{equation*}

\end{Proposition}
\begin{Proposition}\label{prop:Minkowski}
Let $A,B \in \mathcal{C}(\mathbb{R}^n)$, $\lambda \geq 0$, $M\in \mathbb{R}^{m\times n}$
and $l\in \R^n$. \\
Then $\lambda A,\,A+B \in \mathcal{C}(\mathbb{R}^n)$ and $MA\in \mathcal{C}(\mathbb{R}^m)$. Moreover,
\begin{center}
$
\begin{matrix}
 \delta^*(l,\lambda A)=\lambda \delta^*(l,A),& \Y(l,\lambda A)=\lambda \Y(l,A),\\
 \delta^*(l,A+B)=\delta^*(l,A)+\delta^*(l,B),& \Y(l,A+B)=\Y(l,A)+\Y(l,B),\\
 \delta^*(l,M A)=\delta^*(M^T l,A),& \Y(l,M A)=M \Y(M^T l,A).
\end{matrix}
$
\end{center}
\end{Proposition}
By means of the support function or the supporting points, one can fully represent a convex compact set,
either as intersection of halfspaces by the Minkowski duality or as convex hull of supporting points. 
\begin{Proposition}
Let $A\in \mathcal{C}(\mathbb{R}^n)$. Then
\begin{align*}\displaystyle
A&=\bigcap_{l\in S_{n-1}}\big\{x\in \mathbb{R}^n \colon  \scal{l}{x} \le \delta^*(l,A) \big\},\\
A&=\co \bigg( \bigcup_{l\in S_{n-1}}\bb{ y(l,A)} \bigg),\\
\partial A&=\bigcup_{l\in S_{n-1}}\bb{\Y(l,A)},
\end{align*}
where $ y(l,A)$ is an arbitrary selection of $\Y(l,A)$. 
\end{Proposition}
We also recall the definition of Hausdorff distance which is the main tool to measure the error of reachable set approximation.
\begin{Definition}
Let $C,D\in   \mathcal{C}(\mathbb{R}^n),\,x\in \mathbb{R}^n$. Then the \emph{distance function} from $x$ to $D$  is $\di(x,D):=\min_{d\in D} \norm{x-d}$
and the Hausdorff distance between $C$ and $D$ is defined as
\begin{align*}
\dH(C,D)&:=\max \bb{\max_{x\in C} \di(x,D),\max_{y\in D} \di(y,C)},\\
\text{or equivalently}&\\
\dH(C,D)&:= \min \bb{ r \ge 0 \colon C\subset  B_r(D) \text{ and } D\subset  B_r(C) }.
\end{align*}
\end{Definition}
The next proposition will be used for a special form of the space discretization 
of convex sets via the convex hull of finitely many supporting points.
\begin{Proposition}[\mbox{\cite[Proposition 3.4]{BBCG}}]
   Let $A\in \mathcal{C}(\mathbb{R}^n)$, choose $\varepsilon > 0$ with a finite set 
   of normed directions
  \label{prop:approx_conv_set}
    \begin{align*}
       S_{n-1}^{\Delta} & := \bigcup_{k=1,\ldots,N_n} \{l^k\} \subset S_{n-1}
    \end{align*}
   with  $N_n \in \N$, $\dH(S_{n-1},S_{n-1}^{\Delta})\le \varepsilon$ and consider the 
   approximating polytope
    \begin{align*}
       A_\Delta &=\co \bb{ \bigcup_{k=1,\ldots,N_n} \bb{ y(l^k,A) } }
         \subset A,
    \end{align*}
   where $ y(l^k,A)$ is an arbitrary selection of $\Y(l^k,A)$, $k=1,\ldots,N_n$. 
   Then
    \begin{align*}
       \dH(A, A_\Delta) & \leq 2 \diam(A) \cdot \varepsilon,
    \end{align*}
   where $\diam(A)$ stands for the diameter of the set $A$.
\end{Proposition}
Some basic notions of nonsmooth and variational analysis which are needed in constructing and proving the convergence of controls are now introduced. The main references for this part are \cite{CLSW,Rockaf}. 
Let $A$ be a subset in $\R^n$ and $f: A \rightarrow \R \cup \{\infty\}$ be a function. The \emph{indicator function} of $A$ and the \emph{epigraph} of $f$ be defined as 
\begin{equation}
\begin{aligned}
I_A(x)=
\begin{cases}
0 & \quad \text{ if } x\in A\\
+\infty & \quad \text{ otherwise}
\end{cases}, \quad \epi f =\bb{(x,r) \in \R^n \times \R \colon x\in A, \ r\ge f(x)}.
\end{aligned}
\end{equation}
\begin{Proposition}\label{pro:indicator}
Let $A$ be a closed, convex and nonempty set. Then $I_A$ is a lower semicontinuous, 
convex function and $\epi I_A$ is a closed, convex set.
\end{Proposition}
\begin{proof}
see e.g.~\cite[Exercise~2.1]{CLSW}.
\end{proof}
%
%
\begin{Definition}[normal cone and subdifferential in convex case in \mbox{\cite[Sec.~8.C]{Rockaf}}]\label{def:sub}
\mbox{}\\
Let $C \subseteq \R^n$ be a given closed convex set and $f:C \to \mathbb R \cup \{+\infty\}$ be a lower semicontinuous, convex function. Then $v\in \R^n$ is \emph{normal} to $C$ at $x\in C$ if 
\[
\ang{v,y-x} \leq 0\,\,\,\ \forall y\in C.
\]
The set of such vectors is the \emph{normal cone} to $C$ at $x$, denoted by $N_C(x)$. 

{We say that } $v \in \R^n$ is a \emph{subgradient} of $f$ at $x \in \mathrm{dom}(f)$ 
if $(v , -1)$ is an element of the normal cone $N_{\mathrm{epi}(f)}(x, f(x))$.
The possibly empty set of all subgradients of $f$ at $x$, denoted by $\partial f(x)$, is called the \emph{(Moreau-Rockafellar) subdifferential} of $f$ at $x$. 
\end{Definition}
\begin{Definition}[(Painlev\'{e}-Kuratowski) convergence of sets in \mbox{\cite[Sec.~4.A--4.B]{Rockaf}}]\label{def:setconvg}
For a sequence $\bb{A^i}_{i\in \N}$ of subsets of $\R^n$, the \emph{outer limit} is the set
\begin{equation*}
\limsup_{i\rightarrow \infty} A^i=\bb{x\colon \limsup_{i\rightarrow \infty}\di(x,A^i)=0},
\end{equation*}
and the \emph{inner limit} is the set
\begin{equation*}
\liminf_{i\rightarrow \infty} A^i=\bb{x\colon \liminf_{i\rightarrow \infty}\di(x,A^i)=0},
\end{equation*}
The \emph{limit} of the sequence exists if the outer and inner limit sets are equal:
\begin{equation*}
\lim_{i\rightarrow \infty} A^i:=\liminf_{i\rightarrow \infty} A^i=\limsup_{i\rightarrow \infty} A^i.
\end{equation*}
\end{Definition}

We also need two more convergence terms for set-valued maps and functions.

\begin{Definition}[graphical and epi-graphical convergence] \label{def:graph_epi_graph_lim}
Consider $A \subset \R^n$ and the set-valued map $F: A \Rightarrow \R^n$. Then the \emph{graph} of $F$ is defined as
\begin{equation*}
\gph F:=\bb{(x,y) \in \R^n \times  \R^n \colon y\in F(x), \ x \in A}.
\end{equation*}
A sequence of functions $f^i: \R^n \to \R \cup \{\infty\}$, $i\in \N$, converges \emph{epi-graphically},
if the outer and the inner limit of their epigraphs $(\epi f^i)_{i \in \N}$ coincide. The \emph{epi-limit}
is the function for which its epigraph $\epi f$ coincides with the set limit of the epigraphs
in the sense of Painlev\'{e}-Kuratowski (see~\cite[Definition~7.1]{Rockaf}).

We say that the sequence of set-valued maps $(F^i)_{i \in \N}$ with 
$F^i: \R^n \Rightarrow \R^n$ \textit{converges graphically} to a set-valued map 
$F: \R^n \Rightarrow \R^n$ if and only if its graphs, i.e.~the sets $(\gph F^i)_{i \in \N}$, 
converge to $\gph F$ in the sense of Definition \ref{def:setconvg} 
(see~\cite[Definition~5.32]{Rockaf}).
\end{Definition}

We cite here Attouch's theorem in a reduced version which plays an important role for 
convergence results of discrete optimal controls and solutions.

\begin{Theorem}[see~\mbox{\cite[Theorem~12.35]{Rockaf}}] \label{theo:attouch}
   Let $(f^i)_i$ and $f$ be lower semicontinuous, convex, proper functions from $\R^n$ 
   to $\R \cup \{\infty\}$. \\
   Then the epi-convergence of $(f^i)_{i \in \N}$ to $f$ is equivalent to the graphical convergence
   of the subdifferential maps $(\partial f^i)_{i \in \N}$ to $\partial f$.
\end{Theorem}

Now we will recall some basic notations of control theory, see e.g.~\cite[Chap.~IV]{BCD} for more detail. Consider the following linear time-variant control dynamics in $ \mathbb{R}^n$
\begin{equation}\label{LCDyn}
\begin{cases}
\begin{array}{r@{\,}l@{\quad}l}
\dot y(t) & =A(t)y(t)+B(t)u(t)
  & \text{ for a.e. } t \in [t_0, \infty),  \\
 u(t) & \in U 
  & \text{ for a.e. } t \in [t_0, \infty), \\
y(t_0) & = y_0. &
\end{array}
\end{cases}
\end{equation}
The coefficients $ A(t), B(t) $ are $n\times n$ and $n\times m$ matrices respectively,  $y_0 \in \mathbb{R}^n$ is the initial value,
 $U\in \mathcal{C}(\R^m)$ is the set of control values. Under standard assumptions, the existence and uniqueness of \eqref{LCDyn} are guaranteed for any measurable function $u(\cdot)$ and any $y_0 \in \mathbb{R}^n$. Let $\mathcal{S}\subset \mathbb{R}^n$, a nonempty compact set, be the 
\emph{target} and 
$$
\mathcal{U}:=\bb{ u \colon [t_0,\infty) \rightarrow U \text{ measurable}},
$$
the set of \emph{admissible controls}
and $y(t,y_0,u)$ is the solution of \eqref{LCDyn}.

We define the \emph{minimum time starting from $y_0 \in \mathbb{R}^n$ to reach the target $\mathcal{S}$}  for some $u \in \mathcal{U}$ as
$$
t(y_0,u)=\min \,\bb{t\ge t_0:\ y(t,y_0,u)\in \mathcal{S}}\leq \infty.
$$
The \emph{minimum time function to reach $\mathcal{S}$ from $y_0$} is defined as
$$
T(y_0)=\inf_{u\in \mathcal{U}} \,\bb{t(y_0,u)},
$$
see e.g.~\cite[Sec.~IV.1]{BCD}.
We also define the 
\emph{reachable sets for fixed end time} $t> t_0$, \emph{up to time $t$} 
resp.~\emph{up to a finite time} as follows:
 \begin{align*}
    \mathcal{R}(t) & := \bb{y_0 \in \R^n: \textit{ there exists } u\in \mathcal{U},\,y(t,y_0,u)\in \mathcal{S}},\\
    \RSU(t) & := \bigcup_{\substack s \in [t_0,t]}\mathcal{R}(s)= \mbox{}  \bb{y_0 \in \R^n: \textit{ there exists } u\in \mathcal{U},\,y(s,y_0,u)\in \mathcal{S} \text{ for some } s\in [t_0,t]}, \\
    \mathcal{R} & := \bb{y_0 \in \R^n: \textit{ there exists some finite time $ t \geq  t_0$ with }y_0 \in \mathcal{R}(t) } = \bigcup_{t\in [t_0,\infty)} \mathcal{R}(t).
 \end{align*}
By definition 
 \begin{equation} \label{eq:reach_leq_sub_level_set}
    \RSU(t) = \bb{y_0\in \mathbb{R}^n \colon T(y_0)\le t}
 \end{equation}
is a sublevel set
of the minimum time function, while
for a given maximal time $ t_f > t_0$ and some $t \in  I :=  [t_0, t_f]$,
$\mathcal{R}(t)$ is the set of points \emph{reachable from the target in time} $ t $ 
by the \emph{time-reversed system} 
 \begin{align} \label{eq:time_rev_cp}
    \dot{y}(t) & =\bar A(t)y(t) + \bar B(t)u(t), \\
    y(t_0) & \in \mathcal{S}, \label{InCond}
 \end{align}
 where $\bar A(t):= -A(t_0+ t_f-t),\,\bar B(t):=-B(t_0+ t_f-t)$ for shortening notations.
 In other words, $\mathcal{R}(t)$ equals  the set of starting points from 
which the system can reach the target in time $ t $. 
Sometimes $\mathcal{R}(t)$ is called the \emph{backward reachable set} which is also
considered in~\cite{BBZ} for computing the minimum time function by solving
a Hamilton-Jacobi-Bellman equation.

The following standing hypotheses are assumed to be fulfilled in the sequel.
\begin{Assumption}\label{standassum}\mbox{}
\begin{enumerate}
\item[ (i) ] $ A(t),\,B(t) $ are $ n \times n $, $ n \times m $ real-valued matrices defining integrable functions on any compact interval of $[t_0,\infty) $. 
\item[ (ii) ] The control set $U\subset \mathbb{R}^m$ is  convex, compact and nonempty, i.e.~$U \in \mathcal{C}(\R^m)$.
\item[ (iii) ] The target set $\mathcal{S}\in \mathbb{R}^n$ is convex, compact and nonempty, i.e.~$\mathcal{S} \in \mathcal{C}(\R^n)$. \\
   Especially, the target set can be a singleton.
\item[ (iv) ] $\mathcal{R}(t)$ is \emph{strictly expanding} on the compact interval $[t_0, t_f]$, i.e.~$\mathcal{R}(t_1) \subset \inter \mathcal{R}(t_2)$ for all $t_0\le t_1<t_2\le  t_f$
\end{enumerate}
\end{Assumption}
\begin{Remark}
The reader can find sufficient conditions for Assumption \ref{standassum}(iv)  for  $\mathcal{S}=\bb{0}$ in \cite[Chap.~17]{HL}, \cite[Sec.~2.2--2.3]{LM}.  Under this assumption, it is obvious that
\label{rem:strict_expand}
\begin{equation*}
\begin{aligned}
 \mathcal{R}(t)= \RSU(t) .
\end{aligned}
\end{equation*}
\end{Remark}
Under our standard hypotheses, the control problem~\eqref{eq:time_rev_cp} can 
equivalently be replaced by
the following linear differential inclusion
\begin{equation}\label{InLCDyn}
\dot y(t)\in \bar A(t)y(t)+\bar B(t)U \ \ \text{ for a.e. }  t \in [t_0, \infty)
\end{equation}
with absolutely continuous solutions $y(\cdot)$ (see \cite[Appendix~A.4]{Tol}). 

We recall the notion of Aumann's integral \cite{Aum} of a 
set-valued mapping defined as follows.

\begin{Definition}
 Consider $t_f \in [t_0,\infty)$ and the set-valued map $F: [t_0,t_f] \Rightarrow \R^n$ with nonempty images.  
 With the help of the set of integrable selections  
 $$
   \mathcal{F}:= \bb{f \colon [t_0, t_f] \rightarrow \mathbb{R}^n \colon f \text{ is integrable over } [t_0, t_f]  \text{ and } f(t) \in F(t) \text{ for a.e. }  t  \in [t_0, t_f]}
 $$
 the \emph{Aumann's integral} of $F(\cdot)$ is defined as 
 \begin{equation*}
    \int_{t_0}^{ t_f} F(s)ds:=\bb{\int_{t_0}^{ t_f} f(s)ds \colon f \in \mathcal{F} }.
 \end{equation*}
\end{Definition} 

All the solutions of \eqref{InCond}--\eqref{InLCDyn} are represented as
\begin{equation*}
y(t)=\Phi(t,t_0)y_0+\int_{t_0}^{t} \Phi(t,s)\bar B(s)u(s)ds
\end{equation*}
for all $y_0 \in \mathcal{S},\, u\in \mathcal{U}$, and $t_0\le t<\infty$, where $\Phi(t,s)$ is the \emph{fundamental solution matrix} of the homogeneous system 
\begin{equation}\label{fundSol}
\dot y(t)=\bar A(t)y(t), 
\end{equation}
with  $\Phi(s,s)=I_n$, the $n\times n$ identity matrix. 
Using the Minkowski addition and the Aumann's integral, the reachable set can be described by means of Aumann's integral as follows
\begin{equation}\label{Rt}
\mathcal{R}(t)=\Phi(t,t_0)\mathcal{S}+\int_{t_0}^{t} \Phi(t,s)\bar B(s)Uds.
\end{equation}
For time-invariant systems, 
i.e.~$ \bar A(t) = \bar A $, we have $\Phi(t,t_0)=e^{\bar A(t-t_0)}$. 

For the linear control system \eqref{LCDyn} the reachable set at a fixed end time is
convex which allows to apply support functions or supporting points for its approximation.
Furthermore, the reachable sets change continuously with respect to the end time. 
The following theorem will summarize the needed properties.

\begin{Theorem}\label{theo:propertyR}
   Let the Assumptions \ref{standassum}(i)--(iii) be fulfilled and
   consider the linear control process \eqref{LCDyn} in $\mathbb{R}^n$. Then $\mathcal{R}(t)$ is convex, compact and nonempty. Moreover, $\mathcal{R}(t)$ varies continuously with $t_0\le t< \infty$.
\end{Theorem}
\begin{proof}
Recall that for $t\ge t_0$
 \begin{align} \label{eq:sublevel_set} 
      \mathcal{R}(t) =\Phi(t,t_0)\mathcal{S}+\int_{t_0}^{t} \Phi(t,s)\bar B(s)Uds.
 \end{align}
   Observe that the integral term $\int_{t_0}^{t} \Phi(t,s)\bar B(s)Uds$ is actually the reachable set at time $t\ge t_0$ initiating from the origin, it is compact and convex due to the convexity of Aumann's integral, see e.g.~\cite{Aum}. The same properties hold for the
   reachable set $\mathcal{R}(t)$ due to Proposition~\ref{prop:Minkowski} (see also~\cite[Sec.~2.2, Theorem~1]{LM})
    and the assumptions on $\mathcal{S}$. 
      This reference also states the continuity of the set-valued map
      $t \mapsto \mathcal{R}(t)$.
    The proof is completed.
\end{proof}

\begin{Lemma}\label{lem:RSU_contin}
   Under the same assumptions as in Theorem~\ref{theo:propertyR}, the map
   $t \mapsto \RSU(t)$ has nonempty compact images and varies continuously with respect
   to $t \in I  = [t_0,t_f]$.
\end{Lemma}
\begin{proof}
   The integrable linear growth condition holds due to Assumptions~\ref{standassum}(i)
   so that the Filippov-Gronwall theorem in~\cite[Theorem~2.3]{FraRam} 
   applies yielding the compactness of the closure of the set of solutions in the 
   maximum norm on $I$. As a consequence the compactness of $\RSU(t)$ follows easily.

   Let $s,t \in [t_0,t_f]$ and consider $x \in \RSU(s)$. Then, there exists $\widetilde{s}
   \in [t_0,s]$ with $x \in \mathcal{R}(\widetilde{s})$. We distinguish two cases.
   \begin{enumerate}
      \item[case (i):] $\widetilde{s} \leq t$
    \begin{align*}
       \di(x, \RSU(t)) & \leq \di(x, \mathcal{R}(\widetilde{s})) = 0
    \end{align*}
      \item[case (ii):] $\widetilde{s} > t$
    \begin{align*}
       \di(x, \RSU(t)) & \leq \di(x, \mathcal{R}(t)) \leq \sup_{x \in \mathcal{R}(\widetilde{s})} \di(x, \mathcal{R}(t))
         \leq \dH(\mathcal{R}(\widetilde{s}), \mathcal{R}(t))
    \end{align*}
   For every $\varepsilon > 0$ there exists $\delta > 0$ such that 
   for all $s \in [t-\delta, t+\delta] \cap [t_0,t_f]$ we have
    \begin{align*}
       \dH(\mathcal{R}(s), \mathcal{R}(t)) \leq \varepsilon
    \end{align*}
   which also holds for $\widetilde{s}$ instead of $s$, since $0 \leq \widetilde{s} - t \leq s - t$.
   \end{enumerate}
   This proves the continuity of $\RSU(\cdot)$.
\end{proof}

The following proposition is to provide the connection between $\mathcal{R}(t)$ and the level set of $T(\cdot)$ at time $t$ which is essential for this approach. 
We will benefit from
the sublevel representation in~\eqref{eq:reach_leq_sub_level_set}.
 The result is related to~\cite[Theorem~2.3]{BBZ}, where the minimum time function at $x$ is 
the minimum for which $x$ lies on a zero-level set bounding the backward reachable set.


\begin{Proposition}
Let Assumption \ref{standassum} be fulfilled and $t > t_0$. Then 
\label{prop:bd_descr_monotone_case_w_level_set}
\begin{equation}\label{ReachLevel}
\partial \mathcal{R}(t)=\bb{y_0  \in \mathbb{R}^n \colon T(y_0)=t}.
\end{equation}
\end{Proposition}
\begin{proof}
"$\subset$":
Assume that there exists $x\in \partial \mathcal{R}(t)$ with $x\notin \bb{y_0\in \mathbb{R}^n \colon T(y_0)=t}$. 
Clearly, $x \in \RSU(t)$ and~\eqref{eq:reach_leq_sub_level_set} shows that $T(x) \leq t$.
By definition there exists
$s \in [t_0,t]$ with $x \in \mathcal{R}(s)$. Assuming $s <t$ we get the
contradiction $x \in \mathcal{R}(s) \subset \inter \mathcal{R}(t)$ from Assumption~\ref{standassum}(iv).

"$\supset$":
Assume that there exists $x\in \bb{y_0\in \mathbb{R}^n \colon T(y_0)=t}$ (i.e.~$T(x)=t$) be such that $x\notin \partial \mathcal{R}(t)$. Since 
$x\in \mathcal{R}(t)$ by~\eqref{eq:reach_leq_sub_level_set} and we assume that $x\notin \partial \mathcal{R}(t)$, then $x\in \inter( \mathcal{R}(t))$. 

Hence, there exists $\varepsilon > 0$ with
    \begin{align*}
       x + \varepsilon B_1(0) \subset \mathcal{R}(t).
    \end{align*}
   The continuity of $\mathcal{R}(\cdot)$ ensures for $t_1 \in [t - \delta, t+\delta] 
   \cap I$ that
    \begin{align*}
       \dH(\mathcal{R}(t), \mathcal{R}(t_1)) & \leq \frac{\varepsilon}{2}.
    \end{align*}
   Hence,
    \begin{align*}
       x + \varepsilon B_1(0) \subset \mathcal{R}(t) & \subset \mathcal{R}(t_1)
         + \frac{\varepsilon}{2} B_1(0).
    \end{align*}
   The order cancellation law in~\cite[Theorem~3.2.1]{PalUrb}
   can be applied, since $\mathcal{R}(t_1)$ is convex and all sets are compact. Therefore,
    \begin{align*}
       (x + \frac{\varepsilon}{2} B_1(0)) + \frac{\varepsilon}{2} B_1(0) 
         & \subset \mathcal{R}(t_1) + \frac{\varepsilon}{2} B_1(0) \\
       \Rightarrow\ x + \frac{\varepsilon}{2} B_1(0)
         & \subset \mathcal{R}(t_1)
    \end{align*}
   Hence, $x \in \inter( \mathcal{R}(t_1))$ with $t_1 < t$ so that $T(x) \leq t_1 < t$
   which is again a contradiction. Therefore, $\bb{y_0\in \mathbb{R}^n \colon T(y_0)=t} \subset \partial \mathcal{R}(t)$. 
   The proof is completed.
\end{proof}


In the previous characterization of the boundary of the reachable set at fixed end time
the assumption of monotonicity of the reachable sets played a crucial role. 
As stated in Remark~\ref{rem:strict_expand}, Assumption \ref{standassum}(iv) also guarantees that the union of reachable sets 
coincides with the reachable set at the largest end time and is trivially convex.
If we drop this assumption, 
we can only characterize the boundary of the \emph{union} of reachable sets up to a time
under relaxing the expanding property~(iv) while demanding convexity as can be seen in the following proposition.

\begin{Proposition}
Let $t > t_0$, Assumptions \ref{standassum}(i)--(iii) and Assumption
 \begin{quote}
    (iv)' \quad $\RSU(t)$ has convex images and is strictly expanding on the compact interval $[t_0, t_f]$, i.e.
     \begin{align*}
        \RSU(t_1) \subset \inter \RSU(t_2) \quad \text{for all $t_0\le t_1<t_2\le  t_f$}.
     \end{align*}
 \end{quote}
%
Then
\label{prop:bd_descr_w_level_set}
\begin{equation} \label{equ:bd_union_reach_sets}
  \partial \RSU(t)
  = \bb{x\in \R^n \colon T(x)=t}
\end{equation}
\end{Proposition}
 \begin{proof}
  \emph{The proof can be found in~\cite[Proposition~7.1.4]{Le}.}
\end{proof}

\begin{Remark}
   Assumption~(iv)' implies that the considered system is small-time controllable, see~\cite[Chap.~IV, Definition 1.1]{BCD}. Moreover, under the assumption of small-time controllability the nonemptiness of the interior of $\mathcal{R}$ and the continuity of the minimum time
function in $\mathcal{R}$ are consequences, see~\cite[Chap.~IV, Propositions~1.2,~1.6]{BCD}.
Assumption~(iv)' is essentially weaker than~(iv), since the convexity of $\RSU(t)$
and the strict expandedness of $\RSU(\cdot)$ follows by Remark~\ref{rem:strict_expand}. 

In the previous proposition we can allow that $\RSU(t)$ is lower-dimensional and are 
still able to prove the inclusion "$\supset$" in~\eqref{equ:bd_union_reach_sets}, 
since the interior of $\RSU(t)$ would be empty 
and $x$ cannot lie in the interior which also creates the (wanted) contradiction.

For the other inclusion "$\subset$" the nonemptiness of the interior of $\mathcal{R}(t)$
in Proposition~\ref{prop:bd_descr_monotone_case_w_level_set}
resp.~the one of $\RSU(t)$ in Proposition~\ref{prop:bd_descr_w_level_set} is essential. Therefore, the expanding property
in~Assumptions~(iv) resp.~(iv)' cannot be relaxed by assuming only monotonicity in the sense
 \begin{align} \label{ex:relaxed_expand}
    \mathcal{R}(s) \subset \mathcal{R}(t) \quad\text{or}\quad
    \RSU(s) \subset \RSU(t)
 \end{align}
for $s < t$ as \cite[Example~2.6]{BLp2} shows.
\end{Remark}


\section{Approximation of the minimum time function}\label{sec:construction}
\subsection{Set-valued discretization methods}\label{subsec:sv_discr_meth}
Consider the linear control dynamics \eqref{LCDyn}. 

For a given $x\in \mathbb{R}^n$, the problem of computing approximately the minimum time $T(x)$ to reach $\mathcal{S}$ by following the dynamics $\eqref{LCDyn}$ is deeply investigated in literature. It was usually obtained  by solving the associated discrete Hamilton-Jacobi-Bellman equation (HJB), see, for instance, \cite{BF1,F,CL,GL}. Neglecting the space discretization we obtain an approximation of $T(x)$. In this paper, we will introduce another approach to treat this problem based on approximation of the reachable set of the corresponding linear differential inclusion. The approximate minimum time function is not derived from the PDE solver, but from iterative set-valued methods or direct discretization of control problems.

Our aim now is to compute $\mathcal{R}(t)$ numerically up to a maximal time $t_f$ based on the representation~\eqref{Rt} by means of set-valued methods to approximate  Aumann's integral. There are many approaches to achieving this goal. We will describe three known options for discretizing the reachable set which are used in the following. 

Consider for simplicity of notations an equidistant grid over the interval $I=[t_0,t_f]$ with $N$ subintervals, step size $h = \frac{t_f - t_0}{N}$ and grid points $t_i=t_0+ih$, $i=0,\ldots,N$.
\begin{enumerate}
\item[(I)] Set-valued quadrature methods with the exact knowledge of the fundamental solution matrix of \eqref{fundSol} (see e.g.~\cite{V_int,D2F,BL}, \cite[Sec.~2.2]{BPhD}): \\
  As in the pointwise case, we replace the integral $\int_{t_0}^{t} \Phi(t,s)\bar B(s)Uds$ by some quadrature scheme of order $p$ with non-negative weights.
   Therefore, \eqref{Rt} is approximated by
\begin{equation} \label{eq:sv_quad_meth_global}
\mathcal{R}_h(t_N)=\Phi(t_N,t_0)\mathcal{S}+h \sum_{i=0}^{N}c_i \Phi(t_N,t_i)\bar B(t_i)U
\end{equation}
with weights $c_i\ge 0,\,i=0,\ldots,N$. Moreover, the error estimate
\begin{equation*}
\dH(\int_{t_0}^{t_N} \Phi(t_N,s)\bar B(s)Uds,h \sum_{i=0}^{N}c_i \Phi(t_N,t_i)\bar B(t_i)U)\le Ch^p
\end{equation*}
holds. Obviously, the following recursive formula is valid  for $i=0,\ldots,N-1$
\begin{align}
\mathcal{R}_h(t_{i+1}) & =\Phi(t_{i+1},t_{i})\mathcal{R}_h(t_{i})+ h\sum_{j=0}^{1} \widetilde{c}_{ij} \Phi(t_{i+1},t_{i+j})\bar B(t_{i+j})U, \label{eq:sv_quad_meth_local_rec} \\
\mathcal{R}_h(t_0) & = \mathcal{S} \label{eq:sv_quad_meth_local_start}
\end{align}
with suitable weights $\widetilde{c}_{ij} \geq 0$
due to the semigroup property of the fundamental solution matrix, i.e.
$$\Phi(t + s,t_0)=\Phi(t+s,s)\Phi(s,t_0) \quad \text{ for all $t \in I$, $s \ge t_0$
  with $s+t \in I$}.
$$
For example, the set-valued trapezoidal rule uses the settings
 \begin{align*}
    c_0 & = c_N = \frac{1}{2}, \quad c_i = 1 \ \,(i=1,2, \ldots,N-1), \\
    \widetilde{c}_{ij} & = \frac{1}{2} \ \,(i=0,1, \ldots,N-1, \ j=0,1).
 \end{align*}
\item[(II)] Set-valued combination methods (see e.g.~\cite{BL}, \cite[Sec.~2.3]{BPhD}): \\
We replace $\Phi(t_N,t_i)$ in method $ (I) $ by its approximation (e.g.~via ODE solvers of the corresponding matrix equation) such that 
\begin{enumerate}
\item[a)]$\Phi_h(t_{m+n},t_0)=\Phi_h(t_{m+n},t_m)\Phi_h(t_m,t_0)$ for all $m \in \{0,\ldots,N\}$, $n \in \{0,\ldots,N-m\}$ \\
The use of e.g.~Euler's method or Heun's method yields
    \begin{align}
       \Phi_h(t_{i+1},t_i) & = I_n + h A(t_i), \label{eq:euler_ode} \\
       \Phi_h(t_{i+1},t_i) & = I_n + \frac{h}{2} (A(t_i) + A(t_{i+1})) 
         + \frac{h^2}{2} A(t_{i+1}) A(t_i), \label{eq:euler_heun}
    \end{align}
   respectively for $i=0,1,\ldots,N-1$.
\item[b)]$\sup_{0\le i\le N} \norm{\Phi(t_N,t_i)-\Phi_h(t_N,t_i)}\le Ch^p.$
\end{enumerate}
The following global resp.~local recursive approximation together with~\eqref{eq:sv_quad_meth_local_start} holds for the discrete reachable sets:
\begin{align}
\mathcal{R}_h(t_N) & =\Phi_h(t_N,t_0)\mathcal{S}+h\sum_{i=0}^{N}c_i \Phi_h(t_N,t_i)\bar B(t_i)U, \label{eq:sv_comb_meth_global} \\
\mathcal{R}_h(t_{i+1}) & =\Phi_h(t_{i+1},t_{i})\mathcal{R}_h(t_{i})+ h \sum_{j=0}^{1} \widetilde{c}_{ij} \Phi_h(t_{i+1},t_{i+j})\bar B(t_{i + j})U.
\label{semigroupcombmethRh}
\end{align}
\item[(III)] Set-valued Runge-Kutta methods (see e.g.~\cite{DF,W,V,B,BBCG}): \\
We can approximate \eqref{InLCDyn} by set-valued analogues of Runge-Kutta schemes.
The discrete reachable set is computed recursively with the starting condition in~\eqref{eq:sv_quad_meth_local_start}
for the set-valued Euler scheme (see e.g.~\cite{DF}) as
\begin{equation} \label{eq:sv_euler_rec}
\mathcal{R}_h(t_{i+1})=  \Phi_h(t_{i+1},t_i) \mathcal{R}_h(t_i)+hB(t_i)U
\end{equation}
with \eqref{eq:euler_ode} or with \eqref{eq:euler_heun} for the set-valued
Heun's scheme  with piecewise constant selections  (see e.g.~\cite{V}) as
\begin{equation} \label{eq:sv_heun_rec}
\mathcal{R}_h(t_{i+1})= \Phi_h(t_{i+1},t_i) \mathcal{R}_h(t_i)
+\frac{h}{2}\Big( (I + h A(t_{i+1}))B(t_i)+B(t_{i+1})\Big)U.
\end{equation}
For linear differential inclusions these methods can be regarded as perturbed
set-valued combination methods (see~\cite{B}).
\end{enumerate}
Further options are possible, for instance, methods based on Fliess expansion and Volterra series \cite{FER,GK,KK,LV,PV}. 

The purpose of this paper is not to focus on the set-valued numerical schemes themselves, but on the approximative construction of $T(\cdot)$. Thus we just choose the scheme described in (II) and (III) to present our idea from now on. In practice, there are several strategies in control problems to discretize the set of controls $\mathcal{U}$, see e.g.~\cite{BBCG}. Here  we choose a  \emph{piecewise constant} approximation $\mathcal{U}_h$ for the sake of simplicity which corresponds to use only one selection on the subinterval $[t_i,t_{i+1}]$ in the corresponding set-valued quadrature method. 
 This choice is obvious in the approaches \eqref{eq:sv_euler_rec}, \eqref{eq:sv_heun_rec} and e.g.~for set-valued Riemann sums in \eqref{eq:sv_quad_meth_global} or \eqref{eq:sv_comb_meth_global}. In the recursive formulas for the set-valued Riemann sum, this means that $\widetilde{c}_{i1} = 0$ for $i=0,1,\ldots,N-1$.
Recall that from (II) the discrete reachable set reads as follows.
\begin{align*}
\mathcal{R}_h(t_N)=\bb{y\in \mathbb{R}^n \colon \text{ there exists a piecewise constant control } u_h \in \mathcal{U}_h \text{ and } y_0 \in \mathcal{S}\\
\text{ such that } y=\Phi_h(t_N,t_0)y_0+h\sum_{i=0}^{N}c_i \Phi_h(t_N,t_i)\bar B(t_i)u_h(t_i)}
\end{align*}
or equivalently
\begin{equation*}
\mathcal{R}_h(t_N)=\Phi_h(t_N,t_0)\mathcal{S}+h\sum_{i=0}^{N}c_i \Phi_h(t_N,t_i)\bar B(t_i)U.
\end{equation*}
We set
\begin{equation*}
t_h(y_0,y,u_h) = \min \bb{ t_n \colon n\in \mathbb{N},\,\, y=\Phi_h(t_n,t_0)y_0+h\sum_{i=0}^{n}c_i \Phi_h(t_n,t_i)\bar B(t_i) u_h(t_i) }
\end{equation*}
for some  $y\in \mathbb{R}^n,\,y_0\in \mathcal{S}$ and a piecewise constant grid function $u_h$ with $u_h(t_i) = u_i \in U$, $i=0,\ldots,n$. If there does not exist such a grid control $u_h$ which reaches $y$ from $y_0$ by the corresponding discrete trajectory, $t_h(y_0,y,u_h)=\infty$. Then the discrete minimum time function $T_h(\cdot)$ is defined as
\begin{equation*}
T_h(y)=\min\limits_{\substack{ u_h\in \mathcal{U}_h\\[0.3ex] y_0\in \mathcal{S}} }\, t_h(y_0,y,u_h).
\end{equation*}

\begin{Proposition}
In all of the constructions (I)--(III) described above, $\mathcal{R}_h(t_N)$ is a convex, compact and nonempty set.
\end{Proposition}
\begin{proof}
The key idea of the  proof of this proposition is to employ the linearity of \eqref{InLCDyn}, in conjunction with the convexity of $\mathcal{S},\,U$ and Proposition \ref{prop:Minkowski}. In particular, it follows analogously to the proof of \cite[Proposition~3.3]{BBCG}.
\end{proof}
    \begin{Theorem}\label{dHRRh}
Consider the linear control problem \eqref{InCond}--\eqref{InLCDyn}. Assume that the set-valued quadrature method and the ODE solver have the same order $p$. Furthermore, assume that $\bar A(\cdot)$ and $\delta^*(l,\Phi(t_f,\cdot)\bar B(\cdot)U)$ have absolutely continuous $(p-2)$-nd derivative, the $(p-1)$-st derivative is of bounded variation uniformly with respect to all $l\in S_{n-1}$ and $\sum_{i=0}^{N}c_i\norm{B(t_i)U}$ is uniformly bounded for $N\in \mathbb{N}$. Then
\begin{equation}
\dH(\mathcal{R}(t_N),\mathcal{R}_h(t_N))\le Ch^p,
\end{equation}
where $C$ is a non-negative constant.
\end{Theorem}
\begin{proof}
See \cite[Theorem 3.2]{BL}.
\end{proof}
\begin{Remark}
For $p=2$ the requirements of Theorem \ref{dHRRh} are fulfilled if $A(\cdot),\,B(\cdot)$ are absolutely continuous  and $A'(\cdot),\,B'(\cdot)$ are bounded variation (see~\cite{DV}, \cite[Secs.~1.6, 2.3]{BPhD}). 
\end{Remark}
The next subsection is devoted to the full discretization of the reachable set, i.e.~we consider the space discretization as well. Since we will work with supporting points, we do this implicitly by discretizing the set $S_{n-1}$ of normed directions.
 This error will be adapted to the error  of the  set-valued  numerical scheme caused by the time discretization  to preserve its order of convergence with respect to time step size as stated in Theorem~\ref{dHRRh}. Then we will describe in detail the procedure to construct the graph of the minimum time function based on the approximation of the reachable sets.  We will also provide the corresponding overall   error estimate.

\subsection{Implementation and error estimate of the reachable set approximation}\label{subsec:Algorithm}
 For a particular problem, according to its smoothness in an appropriate sense we are first able to choose a difference method with a suitable order, say $O(h^p)$ for some $p>0$, to solve \eqref{fundSol} numerically effectively, for instance Euler scheme, Heun's scheme or Runge-Kutta scheme etc.. Then we approximate Aumann's integral in~\eqref{Rt} by a quadrature formula with the same order, for instance Riemann sum, trapezoid rule, or Simpson's rule etc. to obtain the discrete scheme of the global order $O(h^p)$. 

We implement the set arithmetic operations in~\eqref{semigroupcombmethRh} 
only approximately as indicated in Proposition~\ref{prop:approx_conv_set}
and work with finitely many normed directions 
 \begin{equation} \label{eq:discr_unit_spheres}
    \begin{array}{r@{\,}l@{\,}r@{\,}l@{\,}r@{\,}l}
       S_{\mathcal{R}}^{\Delta} & := \{ & l^k & \,:\, k=1,\ldots,N_{\mathcal{R}} & \} & \subset S_{n-1}, \\
       S_{U}^{\Delta} & := \{ & \eta^r & \,:\, r=1,\ldots,N_U & \} & \subset S_{m-1}
    \end{array}
 \end{equation}
satisfying
 \begin{align*}
    \dH(S_{n-1},S_{\mathcal{R}}^{\Delta}) & \le C h^p, \\
    \dH(S_{m-1},S_{U}^{\Delta}) & \le C h^p
 \end{align*}
to  preserve  the order of the considered scheme  approximating the reachable set.

With this approximation we generate a finite set of supporting points of 
$\mathcal{R}_h(\cdot)$ and with its convex hull the fully discrete reachable set
$\mathcal{R}_{h\Delta}(\cdot)$.  
To reach this target, we also discretize the target set $\mathcal{S}$ and the control set $U$
appearing in~\eqref{eq:sv_quad_meth_local_start} and~\eqref{semigroupcombmethRh}, 
e.g.~along the line of Proposition~\ref{prop:approx_conv_set}:

\begin{equation}\label{SUdelta}
\begin{array}{r@{\,}l@{\,}r@{\,}l@{\,}r@{\,}l}
\widetilde{\mathcal{S}}_{\Delta} & :=\bigcup _  {l^k \in S_{\mathcal{R}}^{\Delta}} & \bb{y(l^k,\mathcal{S})}, & \ \,  \mathcal{S}_{\Delta} := \co(\widetilde{\mathcal{S}}_{\Delta}) \\
\widetilde{U}_{\Delta}& :=\bigcup _  {\eta^r \in S_{U}^{\Delta}} & \bb{y(\eta^r,U)} , & \ \, U_{\Delta} := \co(\widetilde{U}_{\Delta})
\end{array}
\end{equation}
Hence, $\mathcal{S}_{\Delta},\,U_{\Delta}$ are polytopes  approximating 
$\mathcal{S}$ resp.~$U$.

  Let $T_{h\Delta}(\cdot)$ be the fully discrete version of $T(\cdot)$ (it will be 
   defined  later in details). Our aim is to  construct the graph of $T_{h\Delta}(\cdot)$ up to a given time $t_{f}$ based on the knowledge of the reachable set approximation. We divide $[t_0, t_f]$ into $K$ subintervals each of length $\Delta t$. Setting
  $$\Delta t=\frac{t_f-t_0}{K},\,h=\frac{\Delta t}{N},$$
  we have $t_f - t_0 = K N h$ and  compute subsequently the sets of supporting points $Y_{h\Delta}(\Delta t)$,\ldots, $Y_{h\Delta}(t_f)$ by the algorithm described below yielding fully discrete reachable sets
  $\mathcal{R}_{h\Delta}(i \Delta t)$, $i=1,\ldots,K$. Here $K$ decides how many sublevel sets of the graph of $T_{h\Delta}(\cdot)$ we would like to have and $h$ is the step size of the numerical scheme computing $Y_{h\Delta}(i\Delta t)$ starting from $Y_{h\Delta}((i-1)\Delta t)$. 

Due to \eqref{eq:sublevel_set} and \eqref{ReachLevel}, the description of each sublevel set of $T(\cdot)$ can be formulated only with its boundary points, i.e.~the supporting points of the reachable sets at the corresponding time. For the discrete setting, at each step, we will determine the value of $T_{h\Delta}(x)$ for $x \in Y_{h\Delta}(\cdot)$. Therefore, we only store this information for constructing the graph of $T_{h\Delta}(\cdot)$ on the subset $[t_0,t_f]$  of its range.\\ 
\begin{Algorithm}
\label{algorithm}~
\begin{enumerate}
\item[ step  1:] Set $Y_{h\Delta}(t_0)=\widetilde{\mathcal{S}}_{\Delta}$, $\mathcal{R}_{h\Delta}(t_0):= \mathcal{S}_{\Delta}$ as in \eqref{SUdelta}, $i=0$.
\item[ step  2:] Compute $\widetilde{Y}_{h\Delta}(t_{i+1})$ as follows
\begin{align*}
\widetilde{Y}_{h\Delta}(t_{i+1}) & =\Phi_h\big(t_{i+1},t_{i}\big)Y_{h\Delta}\big(t_{i}\big)+h\sum_{j=0}^{N}c_j\Phi_h(t_{i+1},t_{ij})\bar B(t_{ij}) \widetilde U_{\Delta}, \\
  \widetilde{\mathcal{R}}_{h\Delta}(t_{i+1})   & =   \co\big( \widetilde{Y}_{h\Delta}(t_{i+1}) \big), 
\end{align*}
where 
 \begin{align} \label{eq:time_steps}
    t_i & =t_0+i\Delta t,\ t_{ij}=t_i+jh \quad(j=0,1,\ldots,N).
 \end{align}
\item[ step   3:] Compute the set of the supporting points $ \bigcup_{l^k\in S_{\mathcal{R}}^{\Delta}} \bb{y(l^k,\widetilde{ \mathcal{R}}_{h\Delta}(t_{i+1}))}$  and set 
 \begin{align} \label{eq:comp_supp_pts_in_fixed_dir}
    Y_{h\Delta}(t_{i+1}) & =  \bigcup\limits_{l^k\in S_{\mathcal{R}}^{\Delta}} \big\{ y\big(l^k,\widetilde{ \mathcal{R}}_{h\Delta}(t_{i+1})\big)\big)  \big\} 
 \end{align}
 where $ y(l^k,\widetilde{ \mathcal{R}}_{h\Delta}(t_{i+1}))$ is an arbitrary element of $\Y(l^k,\widetilde{ \mathcal{R}}_{h\Delta}(t_{i+1}))$ and set $$\mathcal{R}_{h\Delta}(t_{i+1}):=\co(Y_{h\Delta}(t_{i+1})).$$
\item[step  4:]  If $i<K-1$, set $i=i+1$ and go back to step 2. Otherwise, go to step 5.
\item[step  5:] Construct the graph of $T_{h\Delta}(\cdot)$ by 
the (piecewise) linear interpolation based on the values $t_i$ at the points $Y_{h\Delta}(t_i)$, $i=0,\ldots,K$.
\end{enumerate}
\end{Algorithm}
The algorithm computes the set of vertices $Y_{h\Delta}(t_i)$ of the polygon 
$\mathcal{R}_{h\Delta}(t_i)$ which are supporting points in the directions $l^k \in
S_{\mathcal{R}}^{\Delta}$. 
The
following proposition is the error estimate between the fully discrete reachable set $\mathcal{R}_{h\Delta}(\cdot)$ and   $\mathcal{R}(\cdot)$.  
\begin{Proposition}\label{dHR_deltahl}
 Let Assumptions \ref{standassum}(i)--(iii), together with 
\begin{equation}\label{semiR}
\dH\Big(\mathcal{R}_{h}(t_i),\mathcal{R}(t_i )\Big)\le C_{s} h^p
\end{equation}
  for the set-valued combination method~\eqref{eq:sv_comb_meth_global} in (II), be valid. 
  Furthermore, finitely many directions $S_U^{\Delta},\,S_{\mathcal{R}}^{\Delta} \subset S_{n-1}$ are chosen
  with 
  $$\max(\dH(S_{n-1},S_U^{\Delta}),\dH(S_{n-1},S_{\mathcal{R}}^{\Delta}))\le C_{\Delta} h^p.$$
Then, for $h$ small enough, 
 \begin{equation}\label{dHRRdeltahlglobal}
\begin{aligned}
& \dH\Big(\mathcal{R}_{h\Delta}(t_i),\mathcal{R}_h(t_i )\Big)\le C_{f} h^p,\\
& \dH\Big(\mathcal{R}_{h\Delta}(t_i),\mathcal{R}(t_i)\Big)\le C_{f} h^p,
\end{aligned}
\end{equation}
where $C_s,\,C_{\Delta},\, C_f$ are some positive constants and  $t_i=t_0+i\Delta t,\,i=0,\ldots,K$.
\end{Proposition}
\begin{proof}
  \emph{The proof can be found in~\cite[Proposition~7.2.5]{Le}.}
\end{proof}
\begin{Remark}
If $\mathcal S$ is a singleton, we do not need to discretize the target set. 
The overall error estimate in~\eqref{dHRRdeltahlglobal} even improves in this case, since 
$\dH\big(\widetilde{\mathcal{R}}_{h\Delta}(t_0),\mathcal{R}_h(t_0)\big)=0$. 
\end{Remark}
As we can see in this subsection the convexity of the reachable set plays a vital role. Therefore, this approach  can only be extended to special nonlinear  control systems 
with convex reachable sets.

In the following subsection, we provide the error estimation of $T_{h\Delta}(\cdot)$ obtained by the indicated approach under  Assumptions~\ref{standassum}, the regularity of $T(\cdot)$ and the properties of the numerical approximation.
\subsection{Error estimate of the minimum time function}
 After computing the fully discrete reachable sets  in Subsection~\ref{subsec:Algorithm},  we obtain the values  of $T_{h\Delta}(x)$ 
for all $x\in  \bigcup_{i=0,\ldots,K}  Y_{h\Delta}(t_i)$, $t_i= t_0 +  i\Delta t$. 
For all boundary points $x \in \partial \mathcal{R}_{h\Delta}(t_i)$ and some $i=1,\ldots,K$, we define
 \begin{align}
    T_{h\Delta}(x) & = t_i \text{ for } 
                                       x\in \partial  \mathcal{R}_{h\Delta}(t_i) ,
                                  \label{eq:discr_min_time_bd_pt}
    \intertext{together with the initial condition}
    T_{h\Delta}(x) & = t_0 \ \, \text{ for }  x \in \mathcal{S}_{\Delta}. \nonumber
 \end{align}
 The task is now to define a suitable value of $T_{h\Delta}(x)$ 
 in the computational domain
 $$ 
      \Omega := \bigcup_{i=0,\ldots,K} \mathcal{R}_{h\Delta}(t_i),
 $$ 
 if $x$ is neither a boundary point of reachable sets nor lies inside the target set.
 First we construct a simplicial triangulation 
 $\bb{\Gamma_j}_{j=1,\ldots,M}$ 
 over the set 
 $\Omega \setminus \inter(\mathcal{S})$ of points
 with grid nodes in $ \bigcup_{i=0,\ldots,K} Y_{h\Delta}(t_i) $.
 Hence, 
  \begin{itemize}
     \item $\Gamma_j \subset \R^n$ is a simplex for $j=1,\ldots,M$,
     \item $\Omega \setminus \inter(\mathcal{S}) = \bigcup_{j=1,\ldots,M}\Gamma_j$,
     \item the intersection of two different simplices is either empty
           or a common face
     \item all supporting points in the sets $\bb{Y_{h\Delta}(t_i)}_{i=0,\ldots,K}$
           are vertices of some simplex,
     \item all the vertices of each simplex have to belong either to 
           the fully discrete reachable set
           $\mathcal{R}_{h\Delta}(t_i)$ or to $\mathcal{R}_{h\Delta}(t_{i+1})$ for some $i=0,1,\ldots,K-1$.
  \end{itemize}
 For the triangulation as in Figure~\ref{fig:part_triang},
 we introduce the maximal diameter of simplices as
 $$\Delta_{\Gamma}:= \max_{j=1,\ldots,M} \diam(\Gamma_j) .$$
\begin{figure}[htp]
\begin{center}
\includegraphics[scale=0.4]{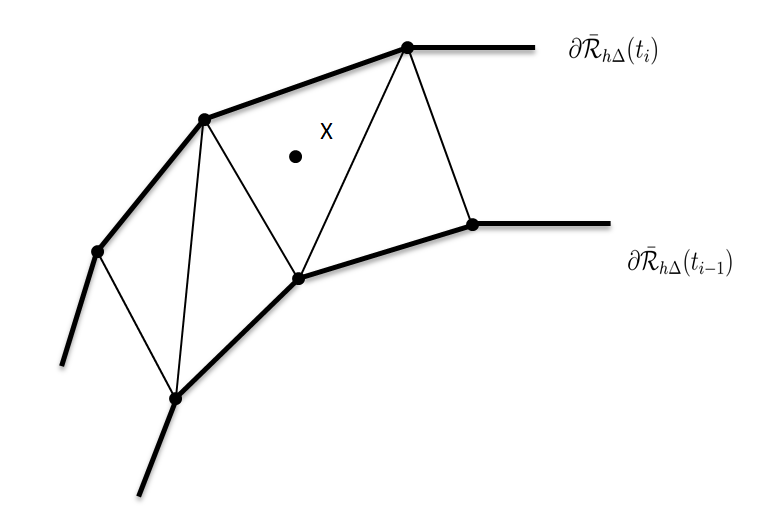}
\end{center}
\caption{part of the triangulation}
\label{fig:part_triang}
\end{figure}\\
 Assume that $x$ is neither a boundary point of one of the computed discrete
     reachable sets $\bb{\mathcal{R}_{h\Delta}(t_i)}_{i=0,\ldots,K}$ nor an element of
     the target set $\mathcal{S}$ and
 let $ \Gamma_j $ be the simplex containing $x$.
 Then 
  \begin{align} \label{eq:cpa_def}
     T_{h\Delta}(x) & =\sum_{ \nu =1}^{ n+1 }\lambda_{ \nu } T_{h\Delta}(x_{ \nu} ),
  \end{align}
 where $x=\sum_{ \nu =1}^{ n+1 }\lambda_{ \nu } x_{ \nu },\,\sum_{ \nu =1}^{\ n+1 }\lambda_{ \nu }=1$  with  $\lambda_{ \nu }\ge 0$
 and $\bb{x_{ \nu }}_{ \nu =1, \ldots,n+1 }$  being the vertices  of $ \Gamma_j $.
  
 If $x$ lies in the interior of $\Gamma_j$, the index $j$ of this simplex is unique.
 Otherwise, $x$ lies on the common face of two or more simplices due to our
 assumptions on the simplicial triangulation and~\eqref{eq:cpa_def} is well-defined. Let $i$ be the index such that $\Gamma_j \in   \mathcal{R}_{h\Delta}(t_i) \setminus  \inter(\mathcal{R}_{h\Delta}(t_{i-1}))$.

 Since $T_{h\Delta}(x_\nu)$ is either $t_i$ or $t_{i-1}$ due to~\eqref{eq:discr_min_time_bd_pt}, we have
  \begin{align*}
     T_{h\Delta}(x) & = \sum_{\nu=1}^{n+1}\lambda_\nu T_{h\Delta}(x_\nu) \le t_i
     \intertext{and}
     \partial \mathcal{R}_{h\Delta}(t_i) & = \bb{y \in \R^n: T_{h\Delta}(y) = t_i}.
  \end{align*}
 The latter holds, since the convex combination is bounded by $t_i$ and equality to $t_i$
 only holds, if all vertices with positive coefficient $\lambda_\nu$ lie on the
 boundary of the reachable set $\mathcal{R}_{h\Delta}(t_i)$.

The following theorem is about the error estimate of the minimum time function obtained by this approach.
\begin{Theorem}\label{errT}
Assume that $T(\cdot)$ is continuous with a non-decreasing modulus $\omega(\cdot)$ in $\mathcal{R}$, i.e.
\begin{equation}
  | T(x)-T(y)|\le \omega(\norm{x-y})\,\,\text{ for all } x,y \in \mathcal{R}.
\end{equation}
Let Assumptions~\ref{standassum} be fulfilled, furthermore assume that
\begin{equation}\label{dHRRhlll}
\dH(\mathcal{R}_{h\Delta}(t),\mathcal{R}(t))\le Ch^p
\end{equation}
holds.  Then
\begin{equation}\label{fullestT}
\norm{T-T_{h\Delta}}_{\infty,\, \Omega  }\le \omega( \Delta_{\Gamma} )+ \omega(Ch^p).
\end{equation}
where $\norm{\cdot}_{\infty,\, \Omega  }$ is the supremum norm taken over $ \Omega  $.
\end{Theorem}
\begin{proof}
We divide the proof into two cases.
\begin{enumerate}
\item[case 1:] $x\in \partial \mathcal{R}_{h\Delta}(t_i)$ for some $i=1,\ldots,K$. \\
Let us choose a best approximation 
$\bar{x} \in \partial\mathcal{R}(t_i)$ of $x$ so that
$$\norm{x-\bar x} = \di(x,\partial \mathcal{R}(t_i)) \leq \dH(\partial \mathcal{R}_{h\Delta}(t_i),\partial \mathcal{R}(t_i))
  =  \dH(\mathcal{R}_{h\Delta}(t_i),\mathcal{R}(t_i)),$$
where we used \cite{Wil} in the latter equality. Clearly, \eqref{ReachLevel}, 
\eqref{eq:cpa_def} show that
$$
  T_{h\Delta}(x)=T(\bar{x})=t_i.
$$
 Then
\begin{align}
|T(x)-T_{h\Delta}(x)|&\le |T(x)-T(\bar{x})|+|T(\bar{x})-T_{h\Delta}(x)| \nonumber \\
& \le \omega(\norm{x-\bar{x}}) \le \omega\big(\dH(\mathcal{R}_{h\Delta}(t_i),\mathcal{R}(t_i))\big)\le \omega(Ch^p) \label{eq:estim_T} 
\end{align} 
due to \eqref{dHRRhlll}.

\item[ case 2:]
$x\in \inter\big(\mathcal{R}_{h\Delta}(t_i)\big)\setminus \mathcal{R}_{h\Delta}(t_{i-1})$
 for some $i=1,\ldots,K$. \\
Let $ \Gamma_j $ be a simplex containing $x$ with 
 the set of  vertices  $\bb{x_j}_{j=1,\ldots, n+1 }$. Then $$T_{h\Delta}(x)=\sum_{j=1}^{ n+1 }\lambda_j T_{h\Delta}(x_j),$$ where $x=\sum_{j=1}^{ n+1 }\lambda_j x_j,\,\sum_{j=1}^{ n+1 } \lambda_j=1, \lambda_j\ge 0$. 
 We obtain
\begin{align*}
&|T(x)-T_{h\Delta}(x)|= |T(x)-\sum_{j=1}^{ n+1 }\lambda_j T_{h\Delta}(x_j)|\\
&\le |T(x)-\sum_{j=1}^{ n+1 }\lambda_j T( x_j)|+|\sum_{j=1}^{ n+1 }\lambda_j T(x_j)-\sum_{j=1}^{ n+1 }\lambda_j T_{h\Delta}(x_j)| \\
&  \le \sum_{j=1}^{\ n+1 }\lambda_j \bigg( |T(x)-T( x_j)|+ |T(x_j) - T_{h\Delta}(x_j)| \bigg)   \le \omega(\Delta_\Gamma)+  \omega(Ch^p),
\end{align*}
where we applied the continuity of $T(\cdot)$ for the first term and the error estimate~\eqref{eq:estim_T} of case~1 for the other.
\end{enumerate}
Combining two cases and noticing that $T(x)=T_{h\Delta}(x)=t_0$ if $x\in \mathcal{S}_{\Delta}$, we get
\begin{equation}
\norm{T-T_{h\Delta}}_{\infty,\,\Omega}
:= \max_{x \in \Omega} |T(x) -T_{h\Delta}(x)|
\le \omega(\Delta_{\Gamma})+ \omega(Ch^p).
\end{equation}
The proof is completed.
\end{proof}
\begin{Remark}\label{Rem_errT}
Theorem  \ref{dHRRh}  provides sufficient conditions for set-valued combination methods such that \eqref{dHRRhlll} holds. 
 See also  e.g.~ \cite{DF} for set-valued Euler's method resp. ~\cite{V} for Heun's method.
If the minimum time function is H\"older continuous  on $\Omega $, \eqref{fullestT} becomes 
\begin{equation}\label{errorHolder}
\norm{T-T_{h\Delta}}_{\infty,\, \Omega  }\le   C\Big((\Delta_{\Gamma} )^{\frac{1}{k}}+  h^{\frac{p}{k}}\Big)
\end{equation}
for some positive constant $C$. The inequality \eqref{errorHolder} shows that the error 
estimate is improved in comparison with the one obtained in~\cite{CL} 
and does not assume explicitly the regularity of optimal solutions as in~\cite{BF2}.
One possibility to define the modulus of continuity satisfying the required  property of non-decrease in Theorem~\ref{errT} is as follows:
\begin{equation*}
\omega(\delta) = \sup \{ \lvert T(x) - T(y)\rvert : \norm{x-y} \le \delta\}
\end{equation*}
An advantage of the methods of Volterra type studied in~\cite{CL} which benefit from
non-standard selection strategies is that the discrete reachable sets converge with higher 
order than 2. The order 2 is an order barrier for set-valued Runge-Kutta methods 
with piecewise constant controls or independent choices of controls, 
since many linear control problems with intervals or boxes for the control values
are not regular enough for higher order approximations (see~\cite{V}).
\end{Remark}
\begin{Remark} There are many different triangulations based on the same data. Among them, we can always choose the one with a smaller diameter close to the Hausdorff distance of the two sets by applying standard grid generators. For example, from the same set of data we can build the two following grids and it is easy to see 
in Figure~\ref{fig:diff_triang}
that the left one 
(for which only three edges are emerging from the corner of the bigger reachable set) 
gives a better approximation, since the maximal diameter in the triangulation at the right
is much bigger.
\begin{figure}[htp]
\begin{center}
\includegraphics[scale=0.4]{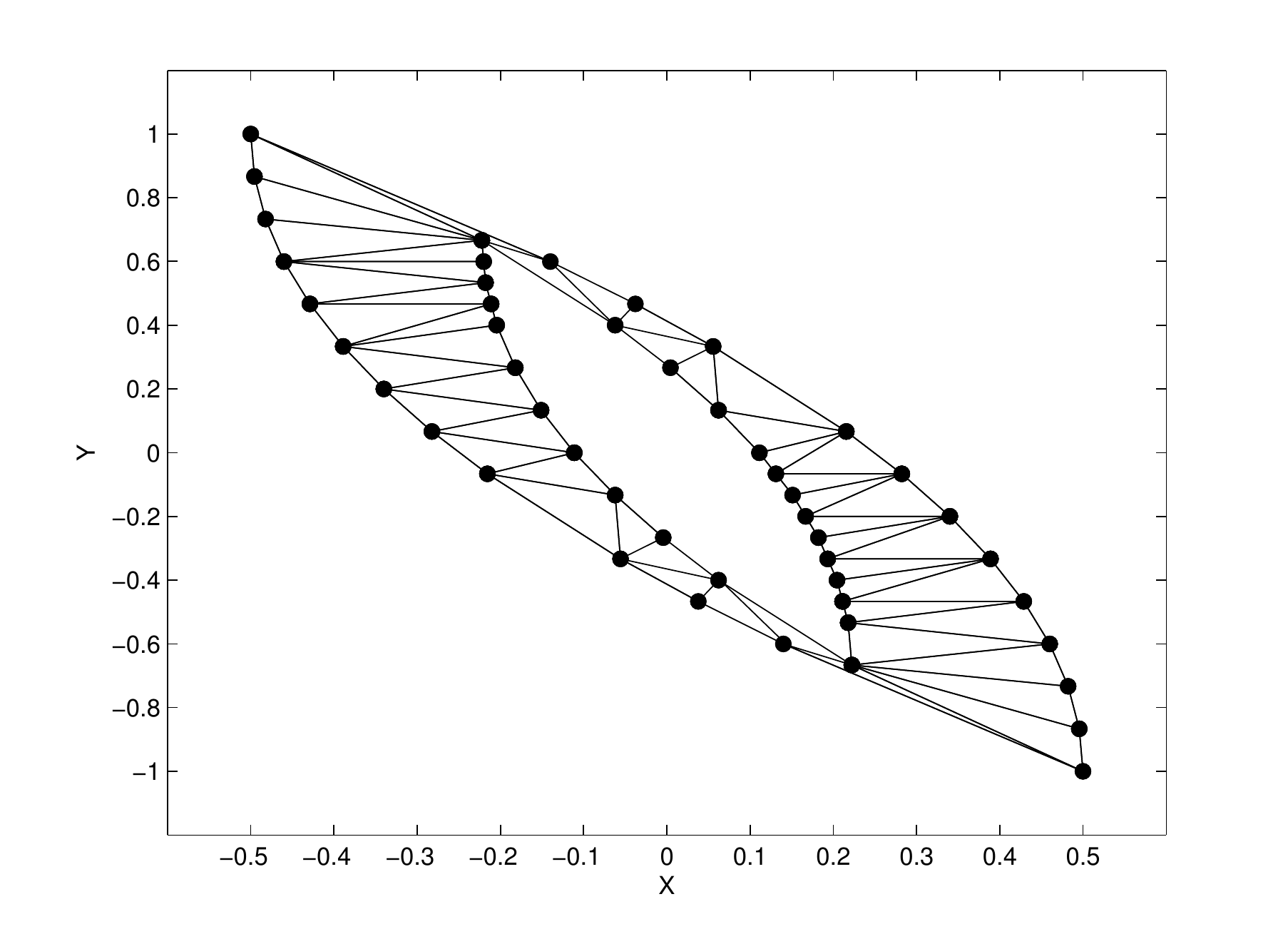}
\includegraphics[scale=0.4]{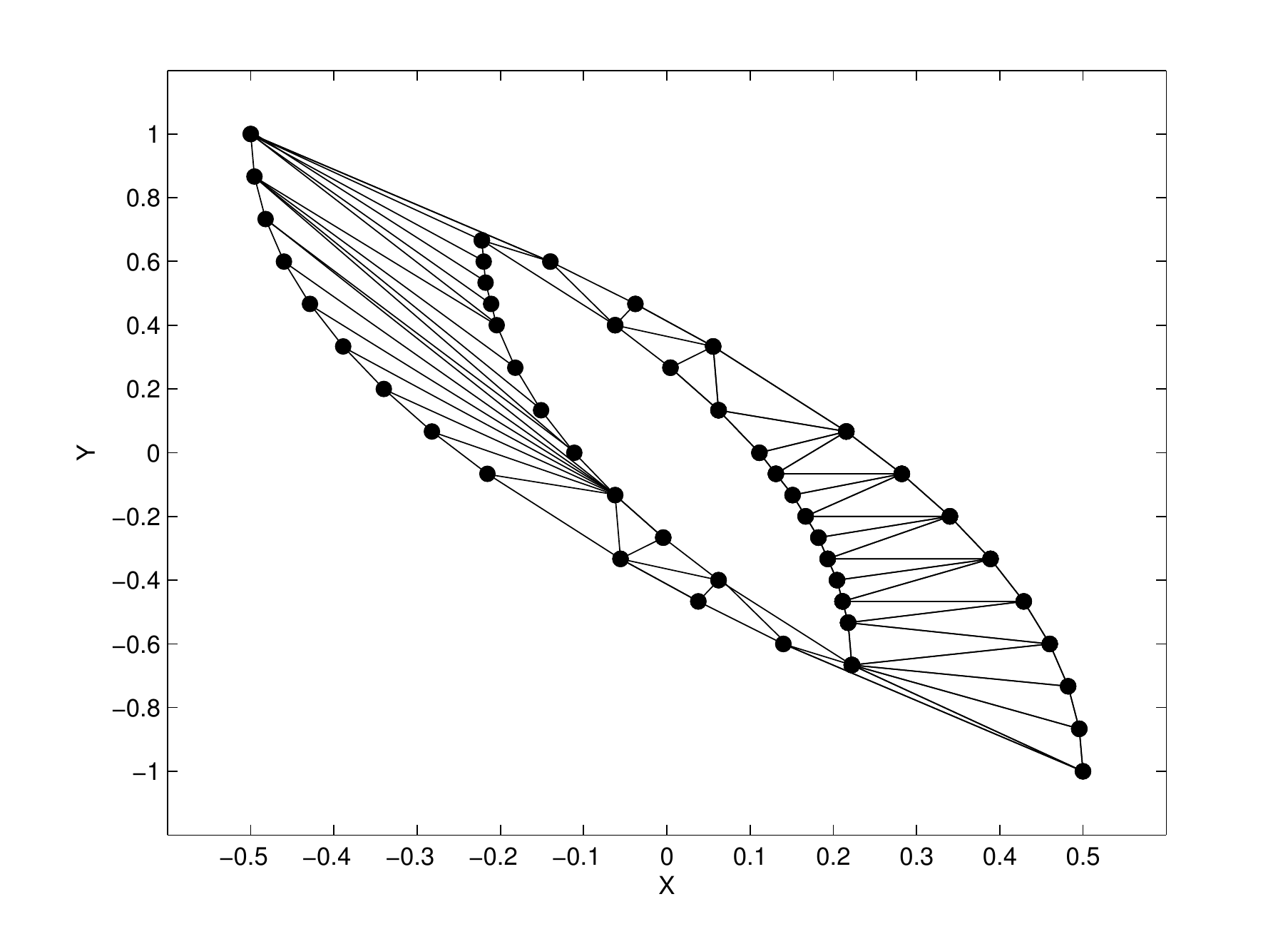}
\end{center}
\caption{two triangulations for the linear interpolation of the minimum time function}
\label{fig:diff_triang}
\end{figure}
\end{Remark}

\begin{Proposition}\label{errTt2}
Let the conditions of Theorem~\ref{errT} be fulfilled.
Furthermore assume that the step size $h$ is so small such that $C h^p$ in~\eqref{dHRRhlll} 
is smaller than $\frac{\varepsilon}{3}$, where
 \begin{align}\label{inclusionR}
    \mathcal{R}(t_i) + \varepsilon B_1(0)
      & \subset \inter \mathcal{R}(t_{i+1}) \quad\mbox{for all $i=0,\ldots,K-1$}.
 \end{align}
Then
 \begin{align}\label{inclusionRh}
     \mathcal{R}_{h\Delta}(t_i) + \frac{\varepsilon}{3} B_1(0)
       & \subset \inter \mathcal{R}_{h\Delta}(t_{i+1})
 \end{align}
and
\begin{equation}\label{fullestTt2}
\norm{T-T_{h\Delta}}_{\infty,\,  \Omega  }\le 2 \Delta t,
\end{equation}
where $\norm{\cdot}_{\infty,\, \Omega  }$ is the supremum norm taken over $ \Omega $.
\end{Proposition}

\begin{proof}

For some $i = 0,\ldots,K-1$ we choose a constant $M_{i+1} > 0$ such that 
$\mathcal{R}(t_{i+1}) \subset M_{i+1} B_1(0)$.
Since $\mathcal{R}(t_i)$ does not intersect the complement of
$ \inter \mathcal{R}(t_{i+1}) $ bounded with $M_{i+1} B_1(0)$ and both are compact sets, 
there exists $\varepsilon > 0$ such that
 \begin{align}
    \mathcal{R}(t_i) + \varepsilon B_1(0) 
      & \subset \inter \mathcal{R}(t_{i+1}) \subset M_{i+1} B_1(0). \label{eq:eps_incl}
 \end{align}
We will show that a similar inclusion as~\eqref{eq:eps_incl} holds for the 
discrete reachable sets for small step sizes. If the step size $h$ is so small
that $C h^p$ in~\eqref{dHRRhlll} is smaller than $\frac{\varepsilon}{3}$, then we have
the following inclusions:
 \begin{align}
     \inter \mathcal{R}(t_{i+1}) 
       & \subset \inter \big( \mathcal{R}_{h\Delta}(t_{i+1}) + C h^p B_1(0) \big)
         = \inter \mathcal{R}_{h\Delta}(t_{i+1}) + C h^p \inter B_1(0),\nonumber \\
     \mathcal{R}(t_i) + \varepsilon B_1(0)
       & \subset \inter \mathcal{R}(t_{i+1}) 
         \subset \inter \mathcal{R}_{h\Delta}(t_{i+1}) + \frac{\varepsilon}{3} B_1(0). \nonumber
     \intertext{By the order cancellation law of convex compact sets in~\cite[Theorem~3.2.1]{PalUrb}}
     \mathcal{R}(t_i) + \frac{2}{3}\varepsilon B_1(0)
       & \subset \inter \mathcal{R}_{h\Delta}(t_{i+1}) \nonumber
     \intertext{and}
     \mathcal{R}_{h\Delta}(t_i) + \frac{\varepsilon}{3} B_1(0)
       & \subset \big( \mathcal{R}(t_i) + \frac{\varepsilon}{3} B_1(0) \big)
         + \frac{\varepsilon}{3} B_1(0)
         \subset \inter \mathcal{R}_{h\Delta}(t_{i+1}) \label{est2levfull}.
 \end{align}
We have
\begin{equation}\label{estcase2b}
|T(x)-T_{h\Delta}(x)|= \sum_{j=1}^{ n+1 }\lambda_j|T(x)- T_{h\Delta}(x_j)|.
\end{equation}
In order to obtain the estimate, we observe that
\begin{enumerate}
\item[1)]   $x_j \in \partial \mathcal{R}_{h\Delta}(t_i)$, then 
$t_\nu \le T(x_j)\le t_{i+1}$ with $\nu=\max\{0, i-1\}$.
\item[2)]    $x\in \inter(\mathcal{R}_{h\Delta}(t_i))\setminus \mathcal{R}_{h\Delta}(t_{i-1})$,
then $t_\nu < T(x)\le t_{i+1}$ with $\nu=\max\{0, i-2\}$.
\end{enumerate}
To prove 1) the inequality
$T(x_j) >= t_0$ is clear. Assume that $T(x_j) < t_{i-1}$ for some $i > 1$. Then $x_j
\in \mathcal{R}(t_{i-1})$. By the estimates~\eqref{dHRRhlll},~\eqref{est2levfull} and $C h^p<\frac{\varepsilon}{3}$, it follows that
 \begin{align*}
    x_j & \in \mathcal{R}_{h\Delta}(t_{i-1}) + C h^p B_1(0) \subset \inter \mathcal{R}_{h\Delta}(t_i)
 \end{align*}
which is a contradiction to the assumption $x_j \in \partial \mathcal{R}_{h\Delta}(t_i)$. Hence, $T(x_j) \geq t_{i-1}$. Assume that $T(x_j) > t_{i+1}$. Then, $x_j \notin \mathcal{R}(t_{i+1})$. Furthermore, $x_j$ cannot be an element of $\mathcal{R}_{h\Delta}(t_i)$, since otherwise
 \begin{align*}
    x_j & \in \mathcal{R}_{h\Delta}(t_i) \subset \mathcal{R}(t_i) + C h^p B_1(0)
      \subset \inter \mathcal{R}(t_{i+1})
 \end{align*}
which is a contradiction to $x_j \notin \mathcal{R}(t_{i+1})$. \\
Therefore, $x_j \notin \mathcal{R}_{h\Delta}(t_i)$  which
contradicts $x_j \in \partial \mathcal{R}_{h\Delta}(t_i)$. Hence, the starting assumption $T(x_j) > t_{i+1}$ must be wrong which proves 
$T(x_j) \leq t_{i+1}$.  \\
To prove 2) 
if we assume $T(x) \leq t_{i-2}$ for some $i \geq 2$, then $x \in \mathcal{R}(t_{i-2})$
and 
 \begin{align*}
    x & \in \mathcal{R}_{h\Delta}(t_{i-2}) + C h^p B_1(0) \subset \inter \mathcal{R}_{h\Delta}(t_{i-1})
 \end{align*}
by estimate~\eqref{dHRRhlll}.
But this contradicts $x \notin \mathcal{R}_{h\Delta}(t_{i-1})$. Therefore, $T(x) > t_{i-2}$.
 
Assuming $T(x) > t_{i+1}$ for some $i < K-1$, then $x \notin \mathcal{R}(t_{i+1})$.
Furthermore, if $x$ is an element of $\mathcal{R}_{h\Delta}(t_i)$,
 \begin{align*}
    x & \in \mathcal{R}_{h\Delta}(t_i) \subset \mathcal{R}(t_i) + C h^p B_1(0)
      \subset \inter \mathcal{R}(t_{i+1})
 \end{align*}
which is a contradiction to $x \notin \mathcal{R}(t_{i+1})$. \\
Therefore, $x \notin \mathcal{R}_{h\Delta}(t_i)$  which
contradicts $x\in \inter(\mathcal{R}_{h\Delta}(t_i))\setminus \mathcal{R}_{h\Delta}(t_{i-1})$.
Hence, the starting assumption $T(x) > t_{i+1}$ must be wrong which proves 
$T(x) \leq t_{i+1}$. Consequently, 1) and 2) are proved.\\
Notice that 
\begin{enumerate}
\item[a)]  the case 1) means
       \begin{alignat*}{2}
          T(x_j) & \in  [t_{i-1}, t_{i+1}] & \quad & (i \geq 1), \\
          T(x_j) &  =t_0 & \quad & (i = 0)
        \end{alignat*}
and $| T(x_j) - T_{h\Delta}(x_j) |  \leq \Delta t $ due to $ T_{h\Delta}(x_j)=t_i,\,i=0,\ldots,K$.
\item[b)] from the case 2), we obtain 
 \begin{align*}
           T(x ) & \in  (t_{i-2}, t_{i+1}] \quad (i \geq 2), \\
           T_{h\Delta}(x_j)& - T(x)   < t_i - t_{i-2} = 2 \Delta t, \\
    T_{h\Delta}(x_j) & - T(x)   > t_{i-1} - t_{i+1} = -2 \Delta t. 
 \end{align*}
Therefore, $    | T(x)  - T_{h\Delta}(x_j) |  \leq 2 \Delta t$
for $i \geq 2$ (similarly with estimates  for $i=0,1$).
\end{enumerate}
Altogether, \eqref{fullestTt2} is proved.
\end{proof}
\subsection{Convergence and reconstruction of discrete optimal trajectories}
In this sub\-section we first prove the convergence of the normal cones of $\mathcal{R}_{h\Delta}(\cdot)$ to the ones of the continuous-time reachable set $\mathcal{R}(\cdot)$ in an appropriate sense. Using this result we will be able to reconstruct discrete optimal trajectories to reach the target from a set of given points and also derive the proof of $L^1$-convergence of discrete optimal controls.
 In the following only convergence under weaker assumptions and no convergence order 1 as in~\cite{ABGL-13} 
are proved (see more references 
therein for the classical field of direct discretization methods). 
We also restrict to linear minimum time problems.\\
The following theorem plays an important role in this reconstruction and will deal with the convergence of the normal cones. If the normal vectors of $\mathcal{R}_{h\Delta}(\cdot)$ converge to the corresponding ones of  $\mathcal{R}(\cdot)$, the discrete optimal controls can be computed with the discrete Pontryagin Maximum Principle under suitable assumptions.\\
For the remaining part of this subsection let us consider a 
fixed index $i\in \bb{1,2\ldots,K}$.
We choose a space discretization $\Delta = \Delta(h)$ with $\Orem{\Delta} = \Orem{h^p}$ 
(compare with~\cite[Sec.~3.1]{BPhD}) and often suppress the index $\Delta$ for the approximate solutions and controls.
\begin{Theorem}\label{theo:normaconver}
Consider a discrete approximation of reachable sets of type (I)--(III) with
 \begin{align} \label{eq:conv_reach_sets}
    \lim_{h \downarrow 0} \dH(\mathcal{R}_{h\Delta}(t_i),\mathcal{R}(t_i)) & = 0.
 \end{align}
Under Assumptions~\ref{standassum}, the set-valued maps $x \mapsto N_{\mathcal{R}_{h\Delta}(t_i)}(x)$ converge graphically to the set-valued map $x \mapsto N_{\mathcal{R}(t_i)}(x)$
for $i=1,\ldots,K$.
\end{Theorem}
\begin{proof}
    Let us recall that, under Assumptions \ref{standassum} and by the construction in Subsec.~\ref{subsec:sv_discr_meth}, $\mathcal{R}_{h\Delta}(t_i)$, $\mathcal{R}(t_i)$ are 
    convex, compact and nonempty sets.
    Moreover, we also have that the indicator functions 
     $I_{\mathcal{R}_{h\Delta}(t_i)}(\cdot),I_{\mathcal{R}(t_i)}(\cdot)$ are lower semicontinuous convex functions (see Proposition~\ref{pro:indicator}). 
By \cite[Example~4.13]{Rockaf} the convergence in~\eqref{eq:conv_reach_sets} with respect to the
Hausdorff set also implies the set convergence in the sense of Definition~\ref{def:setconvg}. 
Hence, \cite[Proposition~7.4(f)]{Rockaf} applies and shows that the corresponding
indicator functions converge epi-graphically. Since the subdifferential of the (convex)
indicator functions coincides with the normal cone by~\cite[Exercise~8.14]{Rockaf},
Attouch's Theorem~\ref{theo:attouch} yields the graphical convergence of the corresponding
normal cones.
\end{proof}
The remainder deals with the reconstruction of discrete optimal trajectories and 
the proof of convergence of optimal controls in the \emph{$L^1$-norm}, 
i.e.~$\int_{0}^{t_i}\|\hat{u}(t)-\hat{u}_h(t) \|_1dt\rightarrow 0$ as $h\downarrow 0$
for $\hat{u}(\cdot),\,\hat{u}_h(\cdot)$ being defined later,  where 
the \emph{$\ell_1$-norm} is defined for $x\in\R^n$ as $ \|x\|_1=\sum_{i=1}^{n}|x_i|$. 
To illustrate the 
idea, we confine to a special form of the target and control set, i.e.~$\mathcal S=\bb{0},\,U=[-1,1]^m,\,t\in [0,t_i]$ 
and the time invariant time-reversed linear system
\begin{align}
\label{eq:invardyn}
\begin{cases}
\dot{y}(t)&=\bar{A}y(t)+\bar Bu(t),\ u(t)\in [-1,1]^m,\\
y(0)&=0.
\end{cases}
\end{align}
Algorithm \ref{algorithm} can be interpreted pointwisely in this context as follows. For any $y_{(i-1)N}\in Y_{h\Delta}(t_i)$ there exists a sequence of controls $\bb{u_{kj}}^{ k=1,\ldots,i-1}_{j=0,\ldots,N}$ such that
\begin{equation}\label{eq:numpointw}
\begin{cases}
y_{(k-1)N}&=\Phi_h\big(t_k,t_{k-1}\big)y_{(k-1)0}+h\sum_{j=0}^{N}c_{kj}\Phi_h(t_k,t_{(k-1)j})\bar Bu_{(k-1)j},\\
y_{00}&=0
\end{cases}
\end{equation}
for $k=1,\ldots,i$. Thus $$y_{(i-1)N}=h\sum_{k=1}^{i}\sum_{j=0}^{N}c_{kj}\Phi_h(t_i,t_{(k-1)j})\bar Bu_{(k-1)j}.$$
The continuous-time adjoint equation of~\eqref{eq:invardyn} written for $n$-row vectors reads as
\begin{equation}\label{eq:adjoint}
\begin{cases}
\dot{\eta}(t)&=- \eta(t)\bar{A}\\
\eta(t_i)&=\zeta,
\end{cases}
\end{equation}
and its discrete version, approximated by the same method (see~\cite[Chap.~5]{Ge}) as the one used 
to discretize \eqref{eq:invardyn}, i.e.~\eqref{eq:numpointw}, can be written as follows. For $k=i-1,i-2,\ldots,0$ and $j=N,N-1,\ldots,1$,
\begin{equation}\label{eq:disadjoint}
\begin{cases}
\eta_{k(j-1)}&= \eta_{kj}  \Phi _h(t_{kj},t_{k(j-1)})  \\
\eta_{(i-1)N}&=\zeta_{h},
\end{cases}
\end{equation}
where $\zeta  ,\,\zeta_h  $ will be clarified later.
By the definition of $t_{kj}$ (see Algorithm~\ref{algorithm}) the index $k0$ can be replaced by $(k-1)N$, the solution of \eqref{eq:disadjoint} in backward time is therefore possible.  Here, the end condition will be chosen subject to certain transversality conditions, see the latter reference for more details. 

 Due to well-known arguments (see e.g.~\cite[Sec.~2.2]{LM}) 
the end point of the time-optimal solution lies on the boundary of the reachable set and the adjoint solution
$\eta(\cdot)$ is an outer normal at this end point.
Similarly, this also holds in the discrete case. The following proposition formulates this fact by a discrete version of \cite[ Sec.~2.2, Theorem~2]{LM}. The proof is just a translation of the one of the cited theorem in \cite{LM} to the discrete language. For the sake of clarity, we will formulate and prove it in detail.
\begin{Proposition}
Consider the system \eqref{eq:invardyn} in $\R^n$  with its adjoint problem~\eqref{eq:adjoint} as well as their discrete pendants \eqref{eq:numpointw}, \eqref{eq:disadjoint} respectively. Let $ \bb{u_{kj}} $  be a sequence of controls, $ \bb{y_{kj}} $ be its corresponding discrete solution. Then under Assumptions \ref{standassum}, for $h$ small enough,  
$y_{(i-1)N} \in Y_{h\Delta}(t_i)$ if and only of there exists nontrivial solution $\bb{\eta_{kj}}$ of \eqref{eq:disadjoint} such that 
$$\eta_{kj}\bar B u_{kj}=\max_{u\in U} \bb{\eta_{kj} \bar Bu}$$
 for $k=0,...,i-1,\,\,j=0,...,N$, where $Y_{h\Delta}(t_i)$ is defined as in Algorithm \ref{algorithm}.
\end{Proposition}
\begin{proof}
Assume that $\bb{u_{jk}}$ is such that $y_{(i-1)N}$ by the response 
$$y_{(i-1)N}=h\sum_{k=1}^{i}\sum_{j=0}^{N}c_{kj}\Phi_h(t_i,t_{(k-1)j})\bar Bu_{(k-1)j}.$$
Since $\mathcal R_{h\Delta}(t_i)$ is a compact and convex set by construction, there exists a supporting hyperplane $\gamma$ to $\mathcal R_{h\Delta}(t_i)$ at $y_{(i-1)N}$. Let $\zeta_h$ be the outer normal vector of $\mathcal R_{h\Delta}(t_i)$ at $y_{(i-1)N}$. Define the nontrivial discrete adjoint response \eqref{eq:disadjoint}, i.e.
\begin{equation*} 
\begin{cases}
\eta_{k(j-1)}&= \eta_{kj}  \Phi _h(t_{kj},t_{k(j-1)})  \\
\eta_{(i-1)N}&=\zeta_{h},
\end{cases}
\end{equation*}
Then $\eta_0=\eta_{(i-1)N} \Phi _h(t_i, 0)=\zeta_h\, \Phi _h(t_i, 0)$. Noticing that $\Phi _h(t_{kj},t_{k(j-1)} )$ is a perturbation of the identity matrix $I_n$, there exists $\bar{h}$ such that $\Phi _h(t_{kj},t_{k(j-1)} )$ is invertible for $h\in [0,\bar{h}]$ and so is $\Phi _h(t_i, 0)$. Therefore, $\eta_{(i-1)N}=\eta_0  \Phi _h^{-1}(t_i, 0)$. Now we compute the inner product of $\eta_{(i-1)N},\,y_{(i-1)N}$:
\begin{equation*}
\begin{aligned}
\eta_{(i-1)N}\,y_{(i-1)N}&= \displaystyle \eta_0  \Phi _h^{-1}(t_i, 0) \Big(h\sum_{k=1}^{i}\sum_{j=0}^{N}c_{kj}\Phi_h(t_i,t_{(k-1)j})\bar Bu_{(k-1)j}\Big)\\
&=h\sum_{k=1}^{i}\sum_{j=0}^{N}c_{kj} \eta_0  \Phi _h^{-1}(t_i, 0) \Phi_h(t_i,t_{(k-1)j})\bar Bu_{(k-1)j}\\
&=h\sum_{k=1}^{i}\sum_{j=0}^{N}c_{kj} \eta_0 \Phi _h^{-1}(t_{(k-1)j}, 0) \Phi _h^{-1}(t_i, t_{(k-1)j})  \Phi_h(t_i,t_{(k-1)j})\bar Bu_{(k-1)j}\\
&=h\sum_{k=1}^{i}\sum_{j=0}^{N}c_{kj} \eta_0 \Phi _h^{-1}(t_{(k-1)j}, 0) \bar Bu_{(k-1)j}=h\sum_{k=1}^{i}\sum_{j=0}^{N}c_{kj} \eta_{(k-1)j} \bar Bu_{(k-1)j}\\
\end{aligned}
\end{equation*}
Now assume that $\eta_{kj}\bar B u_{kj}< \max_{u\in U} \bb{\eta_{kj} \bar Bu}$ for some indices $k,\,j$. Then define another sequence of controls as follows
\begin{equation*}
\tilde{u}_{kj}=
\begin{cases}
u_{kj} &\text{ if } \eta_{kj}\bar B u_{kj}=\max_{u\in U} \bb{\eta_{kj} \bar Bu}\\
\max_{u\in U} \bb{\eta_{kj} \bar Bu} &\text{ otherwise}.
\end{cases}
\end{equation*}
Let $\tilde{y}_{(i-1)N}$ be the end point of the discrete trajectory following $\bb{\tilde{u}_{kj}}$. We have
\begin{equation*}
 \eta_{(i-1)N}\,\tilde{y}_{(i-1)N}=h\sum_{k=1}^{i}\sum_{j=0}^{N}c_{kj} \eta_{(k-1)j} \bar B \tilde u_{(k-1)j}
 \end{equation*}
which implies that 
 $\eta_{(i-1)N}\, y_{(i-1)N}<\eta_{(i-1)N}\,\tilde{y}_{(i-1)N}$ or $\eta_{(i-1)N}( \tilde{y}_{(i-1)N}-y_{(i-1)N})>0$ which contradicts the construction of $\eta_{(i-1)N}=\zeta_h$, an outer normal vector of $\mathcal R_{h\Delta}(t_i)$ at $y_{(i-1)N}$. Therefore, $\eta_{kj}\bar B u_{kj}= \max_{u\in U} \bb{\eta_{kj} \bar Bu}$.\\
 Conversely, assume that for some nontrivial discrete adjoint response $\eta_{(i-1)N}=\eta_0  \Phi _h^{-1}(t_i, 0)$, the controls satisfies
 \begin{equation}\label{eq:asscontrol}
 \eta_{kj}\bar B u_{kj}= \max_{u\in U} \bb{\eta_{kj}\bar B u}
 \end{equation}
  for every indices $k=0,...,i-1,\,j=0,...,N$. We will show that the end point $y_{(i-1)N}$ of the corresponding trajectory $\bb{y_{kj}}$ will lie at the boundary of $\mathcal R_{h\Delta}(t_i)$, not at any point belonging to its interior. Suppose, by contradiction, $y_{(i-1)N}$ lies in the interior of 
 $\mathcal R_{h\Delta}(t_i)$. Let $\tilde{y}_{(i-1)N}$ be a point reached by a sequence of controls $\bb{\tilde{u}_{kj}}$ in $\mathcal R_{h\Delta}(t_i)$ in such that 
 \begin{equation}\label{eq:ineqcontrass}
 \eta_{(i-1)N}y_{(i-1)N} < \eta_{(i-1)N}\tilde{y}_{(i-1)N}.
 \end{equation}
Our assumption \eqref{eq:asscontrol} implies that 
 \begin{equation}\label{eq:ineqcontr}
 \eta_{kj}\bar B \tilde{u}_{kj}\le  \eta_{kj} \bar Bu_{kj}
 \end{equation} for all $k,j$. As above, due to \eqref{eq:ineqcontr}, we show that
 $\eta_{(i-1)N}\tilde{y}_{(i-1)N}\le \eta_{(i-1)N}y_{(i-1)N}$ which is a contradiction to \eqref{eq:ineqcontrass}. Consequently, $y_{(i-1)N} \in \partial \mathcal R_{h\Delta}(t_i)=Y_{h\Delta}(t_i)$.
\end{proof}

Motivated by the outer normality of the adjoints in continuous resp.~discrete time and the
maximum conditions, we 
define the optimal controls $\hat{u}(t),\,\hat{u}_h(t)$ as follows
\begin{equation}\label{def:contr}
\left\{
\begin{aligned}
\hat{u}(t) & =\sign (\eta(t)\bar B )^\top & \ \, & \text{for } (t \in [0,t_i]), \\
\hat{u}_h(t)&=\hat{u}_{kj}  & & \text{if } t\in [t_{kj},t_{k(j+1)}),\,k=0,...,i-1,\, j=0,...,N-1,\\
\hat{u}_h(t_{(i-1)N})&=\hat{u}_{(i-1)(N-1)} & & \text{for } t=t_{(i-1)N},
\end{aligned}
\right.
\end{equation}
where $\hat{u}_{kj}=\sign (\eta_{kj}\bar B)^\top,\,k=0,...,i-1,\,j=0,...,N$ 
and 
\begin{equation*}
w := \sign(v) \text{ with } w_\mu =
\begin{cases}
1 &\text{ if } v_\mu>0,\\
0 &\text{ if } v_\mu=0,\\
-1 &\text{ if } v_\mu<0
\end{cases}
\end{equation*}
is the \emph{signum function} and $v,w \in \R^m$, $\mu=1,\ldots,m$.

Owing to Theorem~\ref{theo:normaconver}, we have that the set-valued maps $(N_{\mathcal{R}_{h\Delta}(t_i)}(\cdot))_h$ converge graphically to $N_{\mathcal{R}(t_i)}(\cdot)$ 
which implies
that for every 
sequence $(y_{(i-1)N},\eta_{(i-1)N})_N$ 
in the graphs there exists an element $(y(t_i),\eta(t_i))$
of the graph such that 
\begin{equation}\label{eq:convernorvec}
(y_{(i-1)N},\eta_{(i-1)N}) \rightarrow (y(t_i),\eta(t_i)) \text{ as } h \downarrow 0,
\end{equation}
where $\eta_{(i-1)N} \in N_{\mathcal{R}_{h\Delta}(t_i)}(y_{(i-1)N}),\,\eta(t_i)\in N_{\mathcal{R}(t_i)}(y(t_i))$. Thus $\zeta,\,\zeta_h$ are chosen such that \eqref{eq:convernorvec} is realized.
%
Then it is obvious that $\eta_{kj} \rightarrow \eta(t_{kj})$ as $h \downarrow 0$ with $k=0,...,i-1$ uniformly in $j=0,...,N$. 

For a function $g \colon I \rightarrow \R^m$, we denote the total variation $V(g,I):=\sum_{1}^{m}V(g_i,I)$, where $V(g_i,I)$ is a usual total variation of the $i$-th components of $g$ over a bounded interval $I\in \R$. 
Now if we assume furthermore that if the system \eqref{eq:invardyn}  
is normal, $\hat{u}_h(t)$ converges to $\hat{u}(t)$ in 
the $L^1$-norm.
\begin{Proposition}
Consider that the minimum time problem with the dynamics~\eqref{eq:invardyn} in $\R^n$. Assume that the normality condition holds, i.e. 
\begin{equation}\label{eq:rank}
\rk \bb{B\omega,AB\omega,\ldots,A^{n-1}B\omega}=n 
\end{equation}
for each (nonzero) vector $\omega$
 along an edge of $U=[-1,1]^m$ or along the two end points of the interval $U=[-1,1]$ if $m=1$. Then, under Assumptions \ref{standassum}, $\int_{0}^{t_{i}} \|\hat{u}(t)-\hat{u}_h(t)\|_1dt \rightarrow 0$ as $h\rightarrow 0$ for any $i\in \bb{1,\ldots,K}$.
\end{Proposition}
\begin{proof}
Due to \eqref{eq:rank} $\hat{u}(t)$ defined as in \eqref{def:contr} on $t_0\le t\le t_i$ is the optimal control to reach the state $\hat{y}(t_i)$ of the corresponding optimal solution from the origin. Moreover, it has a finite number of switchings see \cite[Sec.~2.5, Corollary~2]{LM}. Therefore, the total variation, $V(\hat{u}(t),[t_0,t_i])$, 
 is bounded. Let $\displaystyle I_{kj}=[t_{kj},t_{k(j+1)}),\,\text{ for  }k=0,\ldots,i-1,\,j=0,\ldots,N-1, \text{ and except for } I_{(i-1)(N-1)}=[t_{(i-1)(N-1)},t_{(i-1)N}]$. Then 
\begin{equation}
\begin{aligned}
&\int_{I_{kj}}\|\hat{u}(t)-\hat{u}_h(t)\|_1 dt\le \int_{I_{kj}}(\|\hat{u}(t)-\hat{u}(t_{kj})\|_1+ \|\hat{u}(t_{kj})-\hat{u}_h(t_{kj})\|_1)dt\\
&\le hV(\hat{u}(t),I_{kj})+h\|\sign (\eta(t_{kj})\bar B)^\top-\sign (\eta_{kj}\bar B))^\top \|_1
\end{aligned}
\end{equation}
Taking a sum over $k=0,\ldots,i-1,\, j=0,\ldots,N-1$ we obtain
\begin{equation*}
\int_{t_0}^{t_i}\|\hat{u}(t)-\hat{u}_h(t)\|_1dt \le hV(\hat{u}(t),[t_0,t_i])+h\sum_{k=0}^{i-1}\sum_{j=0}^{N-1} \|\sign ( \eta(t_{kj})\bar B))^\top-\sign (\eta_{kj}\bar B))^\top \|_1.\\
\end{equation*}
Since $\hat{u}(t)$ has a finite number of switchings and 
$\eta_{kj}, \eta(t_{kj})$ are non-trivial with the convergence 
$\eta_{kj} \rightarrow \eta(t_{kj})$ as $h\rightarrow 0$ 
for $k=0,\ldots,i,\, j=0,\ldots,N$, the variation $V(\hat{u}(t),[t_0,t_i])$ and $\sum_{k=0}^{i}\sum_{j=0}^{N-1} \|\sign (\eta(t_{kj})\bar B))^\top-\sign (\eta_{kj}\bar B))^\top \|_1$ are bounded. Therefore, 
\begin{equation*}
\int_{t_0}^{t_i} \|\hat{u}(t)-\hat{u}_h(t) \|_1dt \rightarrow 0 \text{ as } h\rightarrow 0.
\end{equation*}
The proof is completed.
\end{proof}

\section*{Acknowledgements}
The authors want to express their thanks to Giovanni Colombo, especially for pointing us 
to Attouch's theorem, and to Lars Gr\"une.
Both of them supported us with helpful suggestions and motivating questions. 
 They are also grateful to Matthias Gerdts about his comments to optimal control.


\end{document}